\documentclass[reqno]{amsart}
\usepackage{amssymb, latexsym}
\usepackage{hyperref}

\let\Re=\undefined\DeclareMathOperator*{\Re}{Re}
\let\Im=\undefined\DeclareMathOperator*{\Im}{Im}
\def\C{{\mathbb C}}
\def\R{{{\mathbb R}}}
\def\Z{{{\mathbb Z}}}

\DeclareMathOperator*{\diverge}{div}

\newcommand{\qtq}[1]{\quad\text{#1}\quad}

\newcommand{\eps}{{\varepsilon}}
\newcommand{\Oh}{\hbox{\O}}
\newcommand{\Tmax}{{T_{\textrm{max}}}}
\newcommand{\Tmin}{{T_{\textrm{min}}}}
\newcommand{\HLM}{{\mathcal M}}
\newcommand{\hi}{{\textit{hi}}}
\newcommand{\lo}{{\textit{lo}}}
\newcommand{\uhi}{{u_{\hi}}}
\newcommand{\ulo}{{u_{\lo}}}
\newcommand{\pl}{{P_{\lo}}}
\newcommand{\ph}{{P_{\hi}}}
\newcommand{\ef}{{\mathcal F}}
\newcommand{\smudge}{\mathcal S}

\theoremstyle{plain}
\newtheorem{theorem}{Theorem}
\newtheorem{proposition}[theorem]{Proposition}
\newtheorem{lemma}[theorem]{Lemma}
\newtheorem{corollary}[theorem]{Corollary}
\theoremstyle{definition}

\newtheorem{definition}[theorem]{Definition}
\newtheorem{remark}[theorem]{Remark}

\newenvironment{CI}{\begin{list}{{\ $\bullet$\ }}{%
\setlength{\topsep}{0mm}\setlength{\parsep}{0mm}\setlength{\itemsep}{0mm}%
\setlength{\labelwidth}{0mm}\setlength{\itemindent}{0mm}\setlength{\leftmargin}{0mm}%
\setlength{\labelsep}{0mm} }}{\end{list}}

\numberwithin{equation}{section}
\numberwithin{theorem}{section}

\begin{document}

\title[Global well-posedness for the quintic NLS in 3D]{Global well-posedness and scattering for the defocusing quintic NLS in three dimensions}
\author{Rowan Killip}
\address{Department of Mathematics\\UCLA\\Los Angeles, CA 90095}
\author{Monica Vi\c{s}an}
\address{Department of Mathematics\\UCLA\\Los Angeles, CA 90095}

\begin{abstract}
We revisit the proof of global well-posedness and scattering for the defocusing energy-critical NLS in
three space dimensions in light of recent developments.  This result was obtained previously by Colliander, Keel, Staffilani, Takaoka, and Tao \cite{CKSTT:gwp}.
\end{abstract}

\maketitle

\section{Introduction}

The defocusing quintic nonlinear Schr\"odinger equation,
\begin{equation}\label{nls}
i u_t +\Delta u = |u|^4 u,
\end{equation}
describes the evolution of a complex-valued function $u(t,x)$ of spacetime $\R_t\times \R^3_x$.  This evolution conserves \emph{energy}:
\begin{equation}\label{energy}
E(u(t)):=\int_{\R^3} \tfrac{1}{2}|\nabla u (t,x)|^2 + \tfrac{1}{6}|u(t,x)|^6 \,dx.
\end{equation}

By Sobolev embedding, $u(0)$ has finite energy if and only if $u(0)\in\dot H^1_x(\R^3)$, which is the space of initial data that we consider.  This is also a scale-invariant space;
both the class of solutions to \eqref{nls} and the energy are invariant under the scaling symmetry
\begin{equation}\label{scaling}
u(t,x)\mapsto u^\lambda (t,x) :=\lambda^{1/2} u(\lambda^2t, \lambda x).
\end{equation}
For this reason, the equation is termed \emph{energy-critical}.

A function $u: I \times \R^3 \to \C$ on a non-empty time interval $I\ni 0$ is called a \emph{strong solution} to \eqref{nls}
if it lies in the class $C^0_t \dot H^1_x(K \times \R^3) \cap L^{10}_{t,x}(K \times \R^3)$ for all compact $K \subset I$, and obeys the Duhamel formula
\begin{align}\label{old duhamel}
u(t) = e^{it\Delta} u(0) - i \int_{0}^{t} e^{i(t-s)\Delta}|u(s)|^4 u(s) \, ds,
\end{align}
for all $t \in I$.  We say that $u$ is a \emph{maximal-lifespan solution} if the solution cannot be extended (in this class) to any strictly larger interval.

Our main result is a new proof of the following:

\begin{theorem}[Global well-posedness and scattering]\label{T:main}\relax\ Let $u_0\in \dot H^1_x(\R^3)$.  Then there exists a unique global strong solution
$u\in C_t^0\dot H^1_x(\R\times \R^3)$ to \eqref{nls} with initial data $u(0)=u_0$.  Moreover, this solution satisfies
\begin{align}\label{E:STB}
\int_{\R}\int_{\R^3} |u(t,x)|^{10}\, dx\, dt \leq C\bigl(\|u_0\|_{\dot H^1_x}\bigr).
\end{align}
Further, scattering occurs: (i) there exist asymptotic states $u_\pm\in \dot H^1_x$ such that
\begin{align}\label{scattering}
\bigl\|u(t)-e^{it\Delta}u_\pm\bigr\|_{\dot H^1_x} \to 0 \quad \text{as}\quad t\to \pm\infty
\end{align}
and (ii) for any $u_+ \in \dot H^1_x$ $($or $u_- \in \dot H^1_x$$)$ there exists a unique global solution $u$ to \eqref{nls} such that \eqref{scattering} holds.
\end{theorem}

Theorem~\ref{T:main} was proved by Colliander, Keel, Staffilani, Takaoka, and Tao in the ground-breaking paper
\cite{CKSTT:gwp}.  The key point is to prove the spacetime bound \eqref{E:STB}; scattering is an easy
consequence of this.  Note also that the solution described in Theorem~\ref{T:main} is in fact unique in the larger class of $C^0_t \dot H^1_x$ functions obeying
\eqref{old duhamel}; this unconditional uniqueness statement is proved in \cite[\S16]{CKSTT:gwp} by adapting earlier work.

The paper \cite{CKSTT:gwp} advanced the induction on energy technique, introduced by Bourgain in \cite{borg:scatter}, and presaged many recent developments in
dispersive PDE at critical regularity.  The argument may be outlined as follows: (i) If a bound of the form \eqref{E:STB} does not hold, then there must be a minimal
almost-counterexample, that is, a minimal-energy solution with (pre-specified) enormous spacetime norm. (ii) By virtue of its minimality, such a solution must have good
tightness and equi-continuity properties.  (iii) To be consistent with the interaction Morawetz identity such a solution must undergo a dramatic change of (spatial) scale
in a short span of time.  (iv)  Such a rapid change is inconsistent with simultaneous conservation of mass and energy.

As just described, the argument appears to be by contradiction, but this is not the case.  In fact, it is entirely quantitative, showing that in order to achieve such
a large spacetime norm, the solution must have at least a certain amount of energy.  The energy requirement diverges as the spacetime norm diverges and so yields an
effective bound for the function $C$ appearing in \eqref{E:STB}.  This style of argument adapts also to other equations and dimensions; see, for example,
\cite{Nakanishi, RV, tao:gwp radial, thesis:art}.

The downside to the induction on energy argument is its complexity.  It is monolithic, as opposed to modular; the value of a small parameter introduced at the very beginning
of the proof is not determined until the very end.  In recent years, the induction on energy argument has been supplanted by a related contradiction argument that is completely modular
and is much easier to understand; it is not quantitative.

The genesis of this new method comes from the discovery of Keraani, \cite{keraani-l2}, that the estimates underlying
the proof that minimal almost-counterexamples have good tightness/equicontinuity properties can be pushed further to
show that failure of Theorem~\ref{T:main} guarantees the existence of a \emph{minimal} counterexample. This insight was
first applied to the well-posedness problem in an important paper of Kenig and Merle, \cite{KM:NLS}, which considered
the focusing equation with radial data in dimensions three, four, and five.  Subsequent papers (by a wide array of authors)
have greatly refined and expanded this methodology.

In this paper, we revisit the proof of Theorem~\ref{T:main} using this `minimal criminal' approach, which, we believe,
results in significant expository simplification.  We will also endeavour to convey that much of the original argument
lives on, both in spirit and in the technical details, by explicit reference to \cite{CKSTT:gwp} as well as by
maintaining their notations, as much as possible.

In some very striking recent work \cite{Dodson:d>2,Dodson:2D,Dodson:1D}, Dodson has proved the analogue of Theorem~\ref{T:main} for the \emph{mass-critical} nonlinear Schr\"odinger equation
in arbitrary dimension. The most significant difference between \cite{CKSTT:gwp} and the argument presented here comes from the adaptation of some of his ideas (present already in the
first paper \cite{Dodson:d>2}) to the problem \eqref{nls}.  We postpone a fuller discussion of these matters until we have described some of the key steps in the proof.

\subsection{Outline of the proof}\label{SS:outline}

We argue by contradiction.  Simple contraction mapping arguments show that Theorem~\ref{T:main} holds for solutions
with small energy; thus, if the theorem were not to hold there must be a transition energy above which the energy no
longer controls the spacetime norm.  The first step in the argument is to show that there is a minimal counterexample
and that, by virtue of its minimality, this counterexample has good compactness properties.

\begin{definition}[Almost periodicity]\label{D:ap}
A solution $u\in L^\infty_t \dot H^1_x(I\times\R^3)$ to \eqref{nls} is said to be \emph{almost periodic (modulo symmetries)} if there exist functions $N: I \to \R^+$, $x:I\to \R^3$,
and $C: \R^+ \to \R^+$ such that for all $t \in I$ and $\eta > 0$,
\begin{equation}\label{E:ap}
\int_{|x-x(t)| \geq C(\eta)/N(t)} \bigl|\nabla u(t,x)\bigr|^2\, dx + \int_{|\xi| \geq C(\eta) N(t)} |\xi|^{2}\, | \hat u(t,\xi)|^2\, d\xi\leq \eta .
\end{equation}
We refer to the function $N(t)$ as the \emph{frequency scale function} for the solution $u$, to $x(t)$ as the \emph{spatial center function}, and to $C(\eta)$ as the
\emph{modulus of compactness}.
\end{definition}

\begin{remark} \label{R:small freq}
Together with boundedness in $\dot H^1_x$, the tightness plus equicontinuity statement \eqref{E:ap} illustrates that almost
periodicity is equivalent to the (co)compact\-ness of the orbit modulo translation and dilation symmetries.  In particular,
from compactness we see that for each $\eta>0$ there exists $c(\eta)>0$ so that for all $t\in I$,
$$
\int_{|x-x(t)| \leq c(\eta)/N(t)} \bigl|\nabla u(t,x)\bigr|^2\, dx + \int_{|\xi| \leq c(\eta) N(t)} |\xi|^{2}\, | \hat u(t,\xi)|^2\, d\xi\leq \eta .
$$
Similarly, compactness implies
$$
\int_{\R^3} |\nabla u(t,x)|^2\, dx \lesssim_u \int_{\R^3} |u(t,x)|^6\, dx
$$
uniformly for $t\in I$.  This last observation plays the role of Proposition 4.8 in \cite{CKSTT:gwp}.
\end{remark}

With these preliminaries out of the way, we can now describe the first major milestone in the proof of
Theorem~\ref{T:main}:

\begin{theorem}[Reduction to almost periodic solutions, \cite{KM:NLS, Berbec}]\label{T:reduct}
{\hskip 0em plus 2em minus 0em}Suppose Theorem~\ref{T:main} failed.  Then there exists a maximal-lifespan solution $u:I\times\R^3\to \C$ to
\eqref{nls} which is almost periodic and blows up both forward and backward in time in the sense that for all $t_0\in
I$,
$$
\int_{t_0}^{\sup I}\int_{\R^3} |u(t,x)|^{10}\, dx\, dt=\int_{\inf I}^{t_0}\int_{\R^3} |u(t,x)|^{10}\, dx\, dt=\infty.
$$
\end{theorem}

The theorem does not explicitly claim that $u$ is a \emph{minimal} counterexample; nevertheless, this is how it is
constructed and, more importantly, how it is shown to be almost periodic.  In \cite{CKSTT:gwp}, the role of this
theorem is played by Corollary~4.4 (equicontinuity) and Proposition~4.6 (tightness).

A pr\'ecis of the proof of Theorem~\ref{T:reduct} can be found in \cite{KM:NLS}, building on Keraani's method
\cite{keraani-l2}; for complete details see \cite{Berbec} or \cite{ClayNotes}.  Just as for the results from
\cite{CKSTT:gwp} mentioned above, the key ingredients in the proof are improved Strichartz inequalities, which show
that concentration occurs, and perturbation theory, which shows that multiple simultaneous concentrations are
inconsistent with minimality.

Continuity of the flow prevents rapid changes in the modulation parameters $x(t)$ and $N(t)$.  In particular, from
\cite[Corollary~3.6]{KTV} or \cite[Lemma 5.18]{ClayNotes} we have

\begin{lemma}[Local constancy property]\label{L:local const}
Let $u:I\times\R^3\to \C$ be a maximal-lifespan almost periodic solution to \eqref{nls}.  Then there exists a small number $\delta$, depending only on $u$, such that if $t_0 \in I$ then
\begin{align*}
\bigl[t_0 - \delta N(t_0)^{-2}, t_0 + \delta N(t_0)^{-2}\bigr] \subset I
\end{align*}
and
\begin{align*}
N(t) \sim_u N(t_0) \quad \text{whenever} \quad |t-t_0| \leq \delta N(t_0)^{-2}.
\end{align*}
\end{lemma}

We recall next a consequence of the local constancy property; see \cite[Corollary~3.7]{KTV} and \cite[Corollary~5.19]{ClayNotes}.

\begin{corollary}[$N(t)$ at blowup]\label{C:blowup criterion}
Let $u:I\times\R^3\to \C$ be a maximal-lifespan almost periodic solution to \eqref{nls}.  If $T$ is any finite endpoint of $I$, then $N(t) \gtrsim_u |T-t|^{-1/2};$
in particular, $\lim_{t\to T} N(t)=\infty$.
\end{corollary}

Finally, we will need the following result linking the frequency scale function $N(t)$ of an almost periodic solution $u$ and its Strichartz norms:

\begin{lemma}[Spacetime bounds]\label{L:ST-N(t)}
Let $u$ be an almost periodic solution to \eqref{nls} on a time interval $I$.  Then
\begin{align}\label{ST via N(t)}
 \int_I N(t)^2\,dt \lesssim_u \|\nabla u\|_{L_t^q L_x^r(I\times \R^3)}^q \lesssim_u 1+ \int_I N(t)^2\,dt
\end{align}
for all $\frac2q+\frac3r=\frac32$ with $2\leq q<\infty$.
\end{lemma}

\begin{proof}
We recall that Lemma~5.21 in \cite{ClayNotes} shows that
\begin{equation}\label{E:Strich as N(t)}
 \int_I N(t)^2\,dt \lesssim_u \int_I\int_{\R^3} |u(t,x)|^{10}\, dx\, dt  \lesssim_u 1+ \int_I N(t)^2\,dt.
\end{equation}
The second inequality in \eqref{ST via N(t)} follows from the second inequality above and an application of the
Strichartz inequality.  The first inequality follows by the same method used to prove the corresponding result in
\eqref{E:Strich as N(t)}:  The fact that $u\not\equiv0$ ensures that $N(t)^{-2/q} \|\nabla u(t)\|_{L_x^r}$ never
vanishes. Almost periodicity then implies that it is bounded away from zero and the inequality follows.
\end{proof}

Let $u:I\times\R^3\to\C$ be an almost periodic maximal-lifespan solution to \eqref{nls}.  As a direct consequence of
the preceding three results, we can tile the interval $I$ with infinitely many \emph{characteristic intervals} $J_k$,
which have the following properties:
\begin{CI}
\item $N(t)\equiv N_k$ is constant on each $J_k$.
\item $|J_k| \sim_u N_k^{-2}$, uniformly in $k$.
\item $\|\nabla u\|_{L_t^q L_x^r(J_k\times \R^3)} \sim_u 1$, for each $\frac2q+\frac3r=\frac32$ with $2\leq q\leq\infty$ and uniformly in $k$.
\end{CI}
Note that the redefinition of $N(t)$ may necessitate a mild increase in the modulus of compactness.  We may further assume that $0$ marks a boundary between characteristic intervals,
which we do, for expository reasons.

Returning to Theorem~\ref{T:reduct}, a simple rescaling argument (see, for example, the proof of Theorem~3.3 in
\cite{TVZ:sloth}) allows us to additionally assume that $N(t)\geq 1$ at least on half of the interval $I$, say, on $[0, \Tmax)$.
Inspired by \cite{Dodson:d>2}, we further subdivide into two cases dictated by the control given by the interaction Morawetz inequality.
Putting everything together, we obtain

\begin{theorem}[Two special scenarios for blowup]\label{T:enemies}
{\hskip 0em plus 1em minus 0em} Suppose Theorem \ref{T:main} failed.  Then there exists an almost periodic solution $u: [0, \Tmax)\times \R^3 \to \C$, such that
$$
\|u\|_{L^{10}_{t,x}( [0, \Tmax) \times \R^4)} =+\infty
$$
and $[0, \Tmax)=\cup_k J_k$ where $J_k$ are characteristic intervals on which $N(t)\equiv N_k \geq 1$.  Furthermore,
$$
\text{either} \quad \int_0^{T_{\max}} N(t)^{-1}\, dt<\infty \quad\text{or}\quad \int_0^{T_{\max}}  N(t)^{-1}\, dt=\infty.
$$
\end{theorem}

Thus, in order to prove Theorem~\ref{T:main} we just need to preclude the existence of the two types of almost periodic
solution described in Theorem~\ref{T:enemies}.  By analogy with the trichotomies appearing in \cite{KTV,Berbec}, we refer to
the first type of solution as a \emph{rapid low-to-high frequency cascade} and the second as a \emph{quasi-soliton}.

In each case, the key to showing that such solutions do not exist is a fundamentally nonlinear relation obeyed by the
equation.  In the cascade case, it is the conservation of mass; in the quasi-soliton case, it is the interaction
Morawetz identity (a monotonicity formula introduced in \cite{CKSTT:interact}).  Unfortunately, both of these relations
have energy-subcritical scaling and so are not immediately applicable to $L^\infty_t \dot H^1_x$ solutions; additional
control on the low frequencies is required. It is in how this control is achieved that we deviate most from
\cite{CKSTT:gwp}.

The argument in \cite{CKSTT:gwp} relies heavily on the interaction Morawetz identity.  To cope with the non-critical
scaling, a frequency localization is introduced.  This produces error terms which are then controlled by means of a
highly entangled bootstrap argument.  Dodson's paper \cite{Dodson:d>2} also uses a frequency-localized interaction
Morawetz identity; however, the error terms are handled via spacetime estimates that are proved independently of this
identity.  Indeed, the proof of these estimates does not even rely on the defocusing nature of the nonlinearity.

In this paper, we adopt Dodson's strategy (see also \cite{Visan:4D}).  The requisite estimates on the low-frequency
part  of the solution appear in Theorem~\ref{T:LTS}.  It seems to us that this theorem represents the limit of what can
be achieved without the use of intrinsically nonlinear tools such as monotonicity formulae.  The rationale for this
assertion comes from consideration of the focusing equation and is discussed in Remark~\ref{R:no better}.
Nevertheless, Theorem~\ref{T:LTS} does just suffice to treat the error terms in the frequency-localized interaction
Morawetz identity (see Section~\ref{S:IM}), which is then used to preclude quasi-solitons in Section~\ref{S:no
soliton}.

The proof of Theorem~\ref{T:LTS} relies on a type of Strichartz estimate that we have not seen previously.  This estimate, Proposition~\ref{P:MaxS}, has the flavour of a maximal
function in that it controls the worst Littlewood-Paley piece at each moment of time.  The necessity of considering a supremum over frequency projections (as opposed to a sum)
is borne out by an examination of the ground-state solution to the focusing equation; see Remark~\ref{R:no better}.  The proof of this proposition is adapted from the double
Duhamel trick first introduced in \cite[\S14]{CKSTT:gwp}.  The original application of this trick also appears here, namely, as Proposition~\ref{P:DDS}.

The non-existence of cascade solutions is proved in Section~\ref{S:cascade}.  The argument combines the following proposition and
Theorem~\ref{T:LTS} to prove first that the mass is finite and then (to reach a contradiction) that it is zero.  It is equally valid in the focusing case.

\begin{proposition}[No-waste Duhamel formula, \cite{ClayNotes,TVZ:cc}]\label{P:duhamel}
Let $u: [0, \Tmax)\times \R^3 \to \C$ be a solution as in Theorem~\ref{T:enemies}.  Then for all $t\in[0, \Tmax)$,
\begin{align*}
u(t) = i \lim_{T\to\, \Tmax} \int_t^T e^{i(t-s) \Delta}  |u(s)|^4u(s)  \, ds,
\end{align*}
where the limit is to be understood in the weak $\dot H^1_x$ topology.
\end{proposition}

\subsection*{Acknowledgements}
The first author was partially supported by NSF grant DMS-1001531.  The second author was partially supported by the Sloan Foundation and NSF grant DMS-0901166.
This work was completed while the second author was a Harrington Faculty Fellow at the University of Texas at Austin.

%
%
%
%

\section{Notation and useful lemmas}\label{S:notation}
We use the notation $X \lesssim Y$ to indicate that there exists some constant $C$ so that $X \leq CY$.
Similarly, we write $X \sim Y$ if $X \lesssim Y \lesssim X$.  We use subscripts to indicate the dependence of $C$ on additional parameters.
For example, $X \lesssim_u Y$ denotes the assertion that $X \leq C_u Y$ for some $C_u$ depending on $u$.

We will make frequent use of the fractional differential/integral operators $|\nabla|^s$ together with the corresponding homogeneous Sobolev norms:
$$
\|f\|_{\dot H^s_x} := \| |\nabla|^s f \|_{L^2_x } \quad\text{where}\quad \widehat{|\nabla|^sf}(\xi) :=|\xi|^s \hat f (\xi).
$$

We will also need some Littlewood--Paley theory.  Specifically, let $\varphi(\xi)$ be a smooth bump supported in the ball $|\xi| \leq 2$ and equalling one on the ball $|\xi| \leq 1$.
For each dyadic number $N \in 2^\Z$ we define the Littlewood--Paley operators
\begin{align*}
\widehat{P_{\leq N}f}(\xi) :=  \varphi(\xi/N)\hat f (\xi), \qquad \widehat{P_{> N}f}(\xi) :=  (1-\varphi(\xi/N))\hat f (\xi),
\end{align*}
\begin{align*}
\widehat{P_N f}(\xi) :=  (\varphi(\xi/N) - \varphi (2 \xi /N))
\hat f (\xi).
\end{align*}
Similarly, we can define $P_{<N}$, $P_{\geq N}$, and $P_{M < \cdot \leq N} := P_{\leq N} - P_{\leq M}$, whenever $M$ and $N$ are dyadic numbers.  We will frequently write $f_{\leq N}$ for
$P_{\leq N} f$ and similarly for the other operators.

The Littlewood--Paley operators commute with derivative operators, the free propagator, and complex conjugation.  They are self-adjoint and bounded on every $L^p_x$ and $\dot H^s_x$
space for $1 \leq p \leq\infty$ and $s\geq 0$.  They also obey the following Sobolev and Bernstein estimates:
\begin{align*}
\| |\nabla|^{\pm s} P_N f\|_{L^p_x} \sim_s N^{\pm s} \| P_N f \|_{L^p_x}, \qquad \|P_N f\|_{L^q_x} \lesssim_s N^{\frac{3}{p}-\frac{3}{q}} \| P_N f\|_{L^p_x},
\end{align*}
whenever $s \geq 0$ and $1 \leq p \leq q \leq \infty$.

We will frequently denote the nonlinearity in \eqref{nls} by $F(u)$, that is, $F(u):=|u|^4u$.  We will use the notation $\Oh(X)$ to denote a quantity that resembles $X$,
that is, a finite linear combination of terms that look like those in $X$, but possibly with some factors replaced by their complex conjugates and/or restricted to various frequencies.
For example,
\begin{equation*}
F(u+v) = \sum_{j=0}^5 \Oh(u^j v^{5-j}) \quad \text{and}\quad F(u)=F(u_{>N}) + \Oh(u_{\leq N} u^4)  \text{ for any } N>0.
\end{equation*}

We use $L^q_tL^r_x$ to denote the spacetime norm
$$
\|u\|_{L_t^qL_x^r} :=\Bigl(\int_{\R}\Bigl(\int_{\R^3} |u(t,x)|^r dx \Bigr)^{q/r} dt \Bigr)^{1/q},
$$
with the usual modifications when $q$ or $r$ is infinity, or when the domain $\R \times \R^3$ is replaced by some smaller spacetime region.  When $q=r$ we abbreviate $L^q_tL^r_x$ by $L^q_{t,x}$.

Let $e^{it\Delta}$ be the free Schr\"odinger propagator.  In physical space this is given by the formula
$$
e^{it\Delta}f(x) = \frac{1}{(4 \pi i t)^{3/2}} \int_{\R^3} e^{i|x-y|^2/4t} f(y) dy.
$$
In particular, the propagator obeys the \emph{dispersive inequality}
\begin{equation}\label{dispersive ineq}
\|e^{it\Delta}f\|_{L^\infty_x(\R^3)} \lesssim |t|^{-3/2}\|f\|_{L^1_x(\R^3)}
\end{equation}
for all times $t\neq 0$.  As a consequence of this dispersive estimate, one obtains the Strichartz estimates; see, for example, \cite{gv:strichartz, tao:keel, Strichartz}.  The particular version we need is from \cite{CKSTT:gwp}.

\begin{lemma}[Strichartz inequality]\label{L:Strichartz}
Let $I$ be a compact time interval and let $u : I\times\R^3 \rightarrow \C$ be a solution to the forced Schr\"odinger equation
\begin{equation*}
i u_t + \Delta u = G
\end{equation*}
for some function $G$.  Then we have
\begin{equation}\label{E:BesStrich}
\Bigl\{\sum_{N\in 2^\Z}\|\nabla u_N\|^2_{L_t^qL_x^r(I\times\R^3)}\Bigr\}^{1/2} \lesssim \|u(t_0)\|_{\dot H^1_x (\R^3)} + \|\nabla G \|_{L^{\tilde q'}_t L^{\tilde r'}_x (I\times\R^3)}
\end{equation}
for any time $t_0 \in I$ and any exponents $(q,r)$ and $(\tilde q,\tilde r)$ obeying $\tfrac2q+\tfrac3r=\tfrac2{\tilde q}+\tfrac3{\tilde r} =\tfrac32$ and $2\leq q, \tilde q\leq \infty$.
Here, as usual, $p'$ denotes the dual exponent to $p$, that is, $1/p + 1/p' = 1$.
\end{lemma}

Elementary Littlewood--Paley theory shows that \eqref{E:BesStrich} implies
\begin{equation*}
 \|\nabla u\|_{L_t^qL_x^r(I\times\R^3)} \lesssim \|u(t_0)\|_{\dot H^1_x (\R^3)} + \|\nabla G \|_{L^{\tilde q'}_t L^{\tilde r'}_x (I\times\R^3)},
\end{equation*}
which corresponds to the usual Strichartz inequality; however, the Besov variant given above allows us to `Sobolev embed' into $L^\infty_x$:

\begin{lemma}[An endpoint estimate]\label{L:4infty}
For any $u:I\times\R^3\to\R$ we have
\begin{align*}
\|u\|_{L_t^4L_x^\infty(I\times\R^3)} \lesssim \|\nabla u\|_{L_t^\infty L_x^2}^{1/2} \Bigl\{\sum_{N\in 2^\Z}\|\nabla u_N\|^2_{L_t^2L_x^6(I\times\R^3)}\Bigr\}^{1/4}.
\end{align*}
In particular, for any frequency $N>0$,
\begin{align*}
\|u_{\leq N}\|_{L_t^4L_x^\infty(I\times\R^3)} \lesssim \|\nabla u_{\leq N}\|_{L_t^\infty L_x^2}^{1/2} \Bigl\{\sum_{M\leq N}\|\nabla u_M\|^2_{L_t^2L_x^6(I\times\R^3)}\Bigr\}^{1/4}.
\end{align*}
\end{lemma}

\begin{proof}
Using Bernstein's inequality we have,
\begin{align*}
\|u\|_{L_t^4L_x^\infty(I\times\R^3)}^4
&\lesssim \int_I \Bigl\{  \sum_{N\in2^\Z} \|u_N(t)\|_{L_x^\infty} \Bigr\}^4 \,dt\\
&\hspace*{-4em}\lesssim \sum_{N_1\leq N_2\leq N_3\leq N_4}\|u_{N_1}\|_{L^\infty_t L_x^\infty} \|u_{N_2}\|_{L^\infty_t L_x^\infty} \|u_{N_3}\|_{L^2_t L_x^\infty} \|u_{N_4}\|_{L^2_t L_x^\infty} \\
&\hspace*{-4em}\lesssim \sum_{N_1\leq \cdots \leq N_4} \Bigl[\tfrac{N_1N_2}{N_3N_4}\Bigr]^{\frac12} \|\nabla u_{N_1}\|_{L^\infty_t L_x^2} \|\nabla u_{N_2}\|_{L^\infty_t L_x^2}
         \|\nabla u_{N_3}\|_{L^2_t L_x^6} \|\nabla u_{N_4}\|_{L^2_t L_x^6} \\
&\hspace*{-4em}\lesssim \|\nabla u\|_{L^\infty_t L_x^2}^2 \sum_{N_3\leq N_4} \Bigl[\tfrac{N_3}{N_4}\Bigr]^{\frac12} \|\nabla u_{N_3}\|_{L^2_t L_x^6} \|\nabla u_{N_4}\|_{L^2_t L_x^6}.
\end{align*}
All spacetime norms above are over $I\times\R^3$.  The claim now follows from Schur's test.
\end{proof}

%
%
%
%

\section{Maximal Strichartz estimates}\label{S:Maxi}

\begin{proposition}\label{P:MaxS}
Let $(i\partial_t +\Delta)v =F+G$ on a compact interval $[0,T]$.  Then for each $6<q\leq\infty$,
\begin{align*}
\Bigl\| M(t)^{\frac3q-1} \bigl\| P_{M(t)} v(t) \bigr\|_{L^q_x} \Bigr\|_{L^2_t} \lesssim \bigl\||\nabla|^{-\frac12} v \bigr\|_{L^\infty_t L^2_x}
        + \bigl\||\nabla|^{-\frac12} G \bigr\|_{L^2_t L^{6/5}_x} + \|F\|_{L^2_t L^1_x}
\end{align*}
uniformly for all functions $M:[0,T]\to 2^\Z$.  All spacetime norms are over $[0,T]\times\R^3$.
\end{proposition}

It is not difficult to see that the conclusion is weaker than (and has the same scaling as) $|\nabla|^{-1/2} v \in L^2_tL^6_x$.
In fact, if $F\equiv 0$, this stronger result can be deduced immediately from the Strichartz inequality.  However, this argument does
not extend to give a proof of the proposition because $F\in L^1_x$ does not imply $|\nabla|^{-1/2} F \in L_x^{6/5}$.
Indeed, the whole theory of the energy-critical NLS in three dimensions is dogged by the absence of endpoint estimates of this type.

The freedom of choosing an arbitrary function $M(t)$ makes this a maximal function estimate; at each time one can take the supremum over
all choices of the parameter.  Writing maximal functions in this way yields linear operators and so one may use the method of $TT^*$; this is an old idea
dating at least to the work of Kolmogorov and Seliverstov in the 1920s (cf. \cite[Ch. XIII]{ZygmundBook}).  As we will see, the double Duhamel trick, which underlies
the proof of Proposition~\ref{P:MaxS}, is a variant of the $TT^*$ idea.  Specifically, one takes the inner-product between two \emph{different} representations of~$v(t)$.

The double Duhamel trick was introduced in \cite[\S14]{CKSTT:gwp}.  There it was used for a different purpose, namely,
to obtain control over the mass on balls.  This is then used to estimate error terms in the (localized) interaction
Morawetz identity.  We will also need this information and for exactly the same reasons; see \eqref{E:22 bound}.
The following proposition captures the main thrust of \cite[\S14]{CKSTT:gwp}:

\begin{proposition}\label{P:DDS}
Let $(i\partial_t +\Delta)v =F+G$ on a compact interval $[0,T]$ and let
\begin{align}\label{E:smudge}
[\smudge_R v](t,x) := \biggl( (\pi R^2)^{-3/2} \! \int_{\R^3} |v(t,x+y)|^2 e^{-|y|^2/R^2} \,dy \biggr)^{1/2}.
\end{align}
Then for each $0<R<\infty$ and $6<q\leq\infty$,
\begin{align}\label{fake Strich}
R^{\frac12-\frac3q} \bigl\| \smudge_R v \bigr\|_{L^2_t L^q_x} \lesssim \|v\|_{L^\infty_t L^2_x} + \|G\|_{L^2_t L^{6/5}_x} + R^{-\frac12} \|F\|_{L^2_t L^1_x},
\end{align}
where all spacetime norms are over $[0,T]\times\R^3$.
\end{proposition}

We use the letter $\smudge$ for the operator appearing in \eqref{E:smudge} to signify both `smudging' and `square function'.  It is easy to see that the
Gaussian smudging used here could be replaced by other methods without affecting the result; indeed, the analogous estimate in \cite{CKSTT:gwp} averages over balls.
That paper also sets $q=100$ and sums over a lattice rather than integrating in $x$.  As $\smudge v$ is slowly varying, summation and integration yield comparable norms.

To control $\smudge_R v$ we need to estimate some complicated oscillatory (and non-oscillatory) integrals.   By choosing a Gaussian weight, some of the integrals can be
done both quickly and exactly; see the proof of Lemma~\ref{L:kernel}.  Before turning to that subject, we first show how the two propositions are inter-connected.  The proof
of the next lemma also demonstrates how bounds on $\smudge_R$ can be used to deduce analogous results with other weights.

\begin{lemma}\label{L:PM vs SR} Fix $6<q\leq \infty$.  Then
\begin{equation}\label{E:PM vs SR}
\sup_{M>0} M^{\frac3q -1} \bigl\| f_M \bigr\|_{L^{q}_x} \lesssim \sup_{M>0} M^{\frac3q -1} \bigl\| \smudge_{M^{-1}} \bigl(f_M \bigr) \bigr\|_{L^{q}_x}.
\end{equation}
\end{lemma}

\begin{proof}
Let $\tilde P_M=P_{M/2}+P_M+P_{2M}$ denote the fattened Littlewood--Paley projector.  The basic relation $P_M=\tilde P_M P_M$ reduces our goal to showing that
\begin{equation}\label{E:PM vs SR'}
  \sup_{M>0} M^{\frac3q -1} \bigl\| \tilde P_M g \bigr\|_{L^{q}_x} \lesssim \sup_{M>0} M^{\frac3q-1} \bigl\| \smudge_{M^{-1}}  g \bigr\|_{L^{q}_x}
\end{equation}
for general functions $g:\R^3\to\C$, say, $g=f_M$.

Recall that the convolution kernel for $\tilde P_M$ takes the form $M^3\psi(Mx)$ for some Schwartz function $\psi$.  By virtue of its rapid decay,
we can write
$$
 |\psi(x)| \leq \int_0^\infty \pi^{-3/2} e^{-|x|^2/\lambda^2} \,d\mu(\lambda)
$$
where $\mu$ is a positive measure with all moments finite.  Indeed, since $\psi$ is radial one can choose $d\mu(\lambda) = 20 |\psi'(\lambda)|\,d\lambda$.
Thus by the Cauchy--Schwarz inequality,
$$
\bigl| [\tilde P_M g](x) \bigr|^2 \leq \int_{\R^3} |g(x+y)|^2 M^3|\psi(My)| \,dy  \leq \int_0^\infty \bigl|[\smudge_{\lambda M^{-1}} g](x)\bigr|^{2} \lambda^{3}\,d\mu(\lambda).
$$
Applying Minkowski's inequality in $L^{q/2}_x(\R^3)$ then easily yields \eqref{E:PM vs SR'}; indeed, one can take the
constant to be $[\int \lambda^{1+6/q}\,d\mu(\lambda)]^{1/2}$.
\end{proof}

\begin{lemma}\label{L:kernel} For fixed $6<q\leq \infty$, the integral kernel
$$
K_R(\tau,z;s,y;x) := (\pi R^2)^{-3/2} \langle \delta_z,\ e^{i\tau\Delta} e^{-|\cdot-x|^2/R^2} e^{is\Delta} \delta_y \rangle
$$
obeys
\begin{equation}\label{E:GaussOrgy}
\!\sup_{R>0} \int_0^\infty \!\!\!\! \int_0^\infty \!\! R^{2-\frac6q} \|K_R(\tau,z;s,y;x)\|_{L^\infty_{z,y} L^{q/2}_x} f(t+\tau) f(t-s)\,ds\,d\tau  \lesssim \bigl| [\HLM f](t) \bigr|^2,
\end{equation}
where $\HLM$ denotes the Hardy--Littlewood maximal operator and $f:\R\to[0,\infty)$.
\end{lemma}

\begin{proof}
From the exact formula for the propagator,
\begin{align}\label{Kdefiningintegral}
K_R(t,z;s,y;x) &=  \int_{\R^3} \frac{ \exp\{ i|z-x'|^2/4\tau - |x'-x|^2/R^2 + i|x'-y|^2/4s  \} }{(4\pi i \tau)^{3/2} (4\pi i s)^{3/2} (\pi R^2)^{3/2}} \,dx'.
\end{align}
Completing the square and doing the Gaussian integral yields
\begin{align*}
|K_R(t,z;s,y;x)| &= (2\pi)^{-3}\bigl[16s^2\tau^2+R^4(s+\tau)^2\bigr]^{-3/4} \exp\Bigl\{ - \frac{R^2(s+\tau)^2 |x-x^*|^2}{16s^2\tau^2+R^4(s+\tau)^2} \Bigr\}
\end{align*}
where $x^* = (sz+ty)/(s+t)$. One more Gaussian integral then yields
\begin{align*}
\|K_R(\ldots)\|_{L^{q/2}_x} &= (2\pi)^{-3} (2\pi/q)^{3/q} R^{-6/q} |s+\tau|^{-6/q} [16s^2\tau^2+R^4(s+\tau)^2]^{-3/4+3/q}.
\end{align*}
Notice that there is no dependence on $z$ or $y$.  This is due to simultaneous translation and Galilei invariance.  In this way, we deduce that
\begin{equation}
\text{LHS\eqref{E:GaussOrgy}} \lesssim \sup_{R>0} \, \int_0^\infty \!\! \int_0^\infty K^*_q(\alpha,\beta) f(t+R^2\alpha) f(t-R^2\beta)\,d\alpha\,d\beta,
\end{equation}
where we have changed variables to $\alpha=R^{-2}\tau$ and $\beta=R^{-2}s$ and written
$$
K^*_q(\alpha,\beta) := \bigl[\alpha+\beta\bigr]^{-6/q} \bigl[\alpha^2\beta^2 + (\alpha+\beta)^2\bigr]^{-3/4+3/q}.
$$

To finish the proof, we just need to show that $K^*_q$ can be majorized by a convex combination of ($L^1$-normalized) characteristic functions of rectangles
of the form $[0,\ell]\times[0,w]$.  In fact, we can write it exactly as a positive linear combination of such rectangles:
$$
K^*_q(\alpha,\beta) = \int_0^\infty \!\! \int_0^\infty \tfrac{1}{\ell} \chi_{[0,\ell]}(\alpha) \, \tfrac{1}{w} \chi_{[0,w]}(\beta) \, \rho(\ell,w) \,d\ell\,dw
    = \int_\alpha^\infty \!\! \int_\beta^\infty \frac{\rho(\ell,w)\,d\ell\,dw}{\ell w}
$$
where $\rho(\ell,w) := \ell w \partial_\ell\partial_w K^*_q(\ell,w) \geq 0$.  Thus, we just need to check that $\rho\in L^1$.  With a little patience, one finds that
$\rho(\ell,w) \lesssim_q K^*_q(\ell,w)$, which leaves us to integrate the latter over a quadrant. We use polar coordinates, $\ell+iw=r e^{i\theta}$:
\begin{align*}
\int_0^\infty \!\! \int_0^\infty  K^*_q(\ell,w) \,d\ell\,dw
&\lesssim \int_0^{\infty} \!\! \int_0^{\pi/2} r^{-6/q} [r^4\sin^2(2\theta)+r^2]^{-3/4+3/q} \,r\,d\theta\,dr \\
&\lesssim \int_0^{\infty} r^{-1/2} (1+r)^{-3/2+6/q} \,dr \lesssim 1.
\end{align*}
Notice that convergence of the $r$ integral relies on $q>6$.  The estimate for the $\theta$ integral given above is only valid in the range $6<q<12$.
When $q>12$, the correct form is $\int r^{-1/2} (1+r)^{-1}\,dr$ and when $q=12$, it is $\int r^{-1/2} \log(2+r)(1+r)^{-1}\,dr$.  Nevertheless, both of these
integrals are also finite.
\end{proof}

We now have all the necessary ingredients to complete the proofs of Propositions~\ref{P:MaxS} and~\ref{P:DDS}.  We only provide the details for the former
because the two arguments are so similar.  Indeed, the proof of the latter essentially follows by choosing $M(t)\equiv R^{-1}$ and throwing away the Littlewood-Paley
projector $P_{M(t)}$ in the argument we are about to present.

\begin{proof}[Proof of Proposition~\ref{P:MaxS}]  In view of Lemma~\ref{L:PM vs SR} we need to show that
$$
\sup_{M>0} M^{\frac3q -1} \bigl\| \smudge_{M^{-1}} \bigl(P_M v(t)\bigr) \bigr\|_{L^{q}_x}   \in L^2_t([0,T])
$$
(with suitable bounds), where the supremum is taken pointwise in time.

As noted earlier, we will use the double Duhamel trick, which relies on playing two Duhamel formulae off against one another, one from each endpoint of $[0,T]$:
\begin{align}
v(t) &= e^{it\Delta}v(0) - i \int_0^t e^{i(t-s)\Delta} G(s)\,ds - i \int_0^t e^{i(t-s)\Delta} F(s)\,ds  \label{twoDuhamels1} \\
&= e^{-i(T-t)\Delta}v(T) + i \int_t^T e^{-i(\tau-t)\Delta} G(\tau)\,d\tau  + i \int_t^T e^{-i(\tau-t)\Delta} F(\tau)\,d\tau. \label{twoDuhamels2}
\end{align}
The idea is to compute the $L^2_x$ norm of $P_M v(t)$ with respect to the Gaussian measure that defines $[\smudge_{M^{-1}} P_M v](t,x)$ by taking the inner product between these
two representations.  Actually, we deviate slightly from this idea because it is not clear how to estimate a pair of cross-terms.  Our trick for avoiding this is the
following simple fact about vectors in a Hilbert space:
\begin{equation}\label{HS trick}
v=a+b=c+d \quad\implies\quad \|v\|^2 \leq 3\|a\|^2 + 3\|c\|^2 + 2|\langle b,d\rangle|.
\end{equation}
(The numbers are neither optimal nor important.) To prove this, write
$$
\|v\|^2 = \langle a,v\rangle + \langle v,c\rangle - \langle a,c\rangle + \langle b,d\rangle
$$
and then use the Cauchy--Schwarz inequality.

Let us invoke \eqref{HS trick} with $a$ and $c$ representing ($P_M$ applied to) the first two summands in \eqref{twoDuhamels1} and \eqref{twoDuhamels2}, respectively, while $b$ and $d$ represent the summands
which involve $F$.  In this way, we obtain the pointwise statement
\begin{align*}
\Bigl|\bigl[\smudge_{M^{-1}} \bigl(P_M v\bigr)\bigr](t,x)\Bigr|^2 &\lesssim \Bigl| \smudge_{M^{-1}} \Bigl( e^{it\Delta} v_M(0)  - i\!\!\int_0^t\! e^{i(t-s)\Delta} G_M(s)\,ds \Bigr)(x)\Bigr|^2 \\
    &\quad + \Bigl| \smudge_{M^{-1}} \Bigl( e^{-i(T-t)\Delta}v_M(T) + i\!\!\int_t^T\!\! e^{-i(\tau-t)\Delta} G_M(\tau)\,d\tau \Bigr)(x)\Bigr|^2 \\
    &\quad + h_M(t,x),
\end{align*}
where $h_M$ is an abbreviation for
\begin{align*}
 h_M(t,x) :={}& \pi^{-3/2} M^3 \biggl|\biggl\langle \int_t^T e^{-i(\tau-t)\Delta} F_M(\tau)\,d\tau,e^{-M^2|\cdot-x|^2} \int_0^t e^{i(t-s)\Delta} F_M(s)\,ds\biggr\rangle\biggr|
\end{align*}

The contributions of the first two summands are easily estimated: For any function $w$, Young's and Bernstein's inequalities imply
\begin{align*}
M^{\frac3q-1} \bigl\| [\smudge_{M^{-1}}(P_M w)](t,x) \bigr\|_{L^q_x(\R^3)} &\lesssim M^{-\frac12} \bigl\| P_M w(t) \bigr\|_{L^6_x(\R^3)}
    \lesssim \bigl\| |\nabla|^{-\frac12} w(t) \bigr\|_{L^6_x(\R^3)}.
\end{align*}
This can then be combined with Strichartz inequality, which shows
$$
\Bigl\| |\nabla|^{-\frac12} \Bigl( e^{it\Delta} v(0)  - i \!\int_0^t\! e^{i(t-s)\Delta} G(s)\,ds \Bigr) \Bigr\|_{L^2_t L^6_x}
    \lesssim \bigl\| |\nabla|^{-\frac12} v(0) \bigr\|_{L^2_x} + \bigl\||\nabla|^{-\frac12}G\bigr\|_{L^2_tL^{6/5}_x}
$$
and similarly for the second summand.

The third summand, $h_M$, is the crux of the matter.  Using the notation from Lemma~\ref{L:kernel} and changing variables, we have
\begin{align*}
h_M(t,x) &=\biggl| \int_0^{T-t}\!\!\!\int_0^t \!\! \iint\! \bar F_M(t+\tau',z) K_{M^{-1}}(\tau',z;s',y;x) F_M(t-s',y)\,dy\,dz\,ds'\,d\tau' \biggr|.
\end{align*}
Note also that by Bernstein's inequality and the maximal inequality,
$$
f(t) := \|F(t)\|_{L^1_x} \qtq{obeys} \|F_M(t)\|_{L^1_x} \lesssim f(t) \qtq{and} \| \HLM f \|_{L^2_t} \lesssim \|F\|_{L^2_t L^1_x}.
$$
Thus using Lemma~\ref{L:kernel} (with $f$ as just defined), we obtain
\begin{align*}
\Bigl\| \sup_{M>0} M^{\frac6q-2} \| h_M(t) \|_{L^{q/2}_x} \Bigr\|_{L^1_t} &\lesssim \|F\|_{L^2_t L^1_x}^2.
\end{align*}
Recalling that $h_M$ appears in an upper bound on the \emph{square} of the size of $P_M v$, the proposition follows.
\end{proof}

%
%
%
%

\section{Long-time Strichartz estimates}\label{S:LTS}

The main result of this section is a long-time Strichartz estimate.  As will be evident from the proof, the result is also valid for $L_t^\infty \dot H^1_x(\R^3)$ solutions
to the focusing equation; see also Remark~\ref{R:no better} at the end of this section.

\begin{theorem}[Long-time Strichartz estimate]\label{T:LTS}
Let $u:(\Tmin,\Tmax)\times\R^3\to \C$ be a maximal-lifespan almost periodic solution to \eqref{nls} and $I\subset(\Tmin,\Tmax)$ a time interval that is tiled by finitely
many characteristic intervals $J_k$.  Then for any fixed $6<q<\infty$ and any frequency $N>0$,
\begin{align}\label{A}
A(N):=\Bigl\{ \sum_{M\leq N} \|\nabla u_{M}\|_{L_t^2 L_x^6(I\times\R^3)}^2\Bigr\}^{1/2}
\end{align}
and
\begin{align}\label{tilde A}
\tilde A_q(N):= N^{3/2}\Bigl\| \sup_{M\geq N} M^{\frac3q-1} \bigl\| u_M (t) \bigr\|_{L^q_x(\R^3)} \Bigr\|_{L^2_t(I)}
\end{align}
obey
\begin{align}\label{E:A bound}
A(N) + \tilde A_q(N)&\lesssim_u 1+ N^{3/2} K^{1/2},
\end{align}
where $K:=\int_I N(t)^{-1}\, dt$.  The implicit constant is independent of the interval $I$.
\end{theorem}

The proof of this theorem will occupy the remainder of this section.  Throughout, we consider a single interval $I$ and so the implicit dependence of $A(N)$, $\tilde A_q(N)$, and $K$
on the interval should not cause confusion.  Additionally, all spacetime norms will be on $I\times\R^3$, unless specified otherwise.

By Bernstein's inequality, $\tilde A_q(N)$ is monotone in $q$.  Thus $q=\infty$ is also allowed.

The analogue of Theorem~\ref{T:LTS} in \cite{CKSTT:gwp} is Proposition~12.1.  Our proof is very different and is inspired by Dodson's work, \cite{Dodson:d>2},
on the mass-critical NLS (see also \cite{Visan:4D}).  In \cite{CKSTT:gwp}, this estimate is derived on the assumption that $u_{>N}$ obeys certain $L^4_{t,x}$
spacetime bounds.  That the solution does admit these spacetime bounds is derived from the interaction Morawetz estimate, using the analogue of \eqref{E:A bound}
to control certain error terms.  This results in a tangled bootstrap argument across several sections of the paper.  The argument that follows does not use the
Morawetz identity, merely Strichartz and maximal Strichartz estimates, and so is equally valid in the focusing case.  We also contend that it is simpler.

The attentive reader will discover that the implicit constant in \eqref{E:A bound} depends only on $u$ through its
$L^\infty_t \dot H^1_x$ norm and its modulus of compactness (cf. Definition~\ref{D:ap}).  Indeed, the dependence on the
latter can be traced to the following: Let $\eta>0$ be a small parameter to be chosen later. Then, by
Remark~\ref{R:small freq} and Sobolev embedding, there exists $c=c(\eta)$ such that
\begin{align}\label{c_0}
\| u_{\leq c N(t)}\|_{L_t^\infty L^6_x} + \|\nabla u_{\leq c N(t)}\|_{L_t^\infty L^2_x}\leq \eta.
\end{align}

By elementary manipulations with the square function estimate and Lemma~\ref{L:4infty}, respectively, we have
\begin{align}\label{duh}
\!\!\! \|\nabla u_{\leq N}\|_{L_t^2 L_x^6}\lesssim A(N),
    \quad
    \| u_{\leq N}\|_{L_t^4 L_x^\infty}\lesssim A(N)^{1/2} \|\nabla u_{\leq N}\|_{L^\infty_t L^2_x}^{1/2}\lesssim_u A(N)^{1/2}.
\end{align}
As noted earlier, the only reason for considering the Besov-type norm that appears in \eqref{A}, rather than the simpler $L_t^2 L_x^6$ norm, is that it allows us to deduce
these $L^4_t L^\infty_x$ bounds.

By combining the Strichartz inequality (Lemma~\ref{L:Strichartz}) with Lemma~\ref{L:ST-N(t)} we have
\begin{align}\label{E:A finite}
A(N)^2 \lesssim_u 1 + \int_I N(t)^2\, dt \lesssim_u \int_I N(t)^2\, dt.
\end{align}
Note that the second inequality relies on the fact that $I$ contains at least one whole characteristic interval $J_k$.  Similarly, using Proposition~\ref{P:MaxS} and then Bernstein's inequality we find
\begin{align*}
\tilde A_q(N)&\lesssim N^{3/2} \Bigl\{ \bigl\| |\nabla|^{-1/2} u_{\geq N}\bigr\|_{L_t^\infty L_x^2} + \bigl\| |\nabla|^{-1/2} P_{\geq N} F(u)\bigr\|_{L_t^2 L_x^{6/5}}  \Bigr\}\\
&\lesssim 1+ \|\nabla u\|_{L_t^2 L_x^6} \|u\|_{L_t^\infty L_x^6}^4\\
&\lesssim_u \Bigl( \int_I N(t)^2\, dt \Bigr)^{1/2}.
\end{align*}
Thus
\begin{align}\label{E:A:N large}
A(N) + \tilde A_q(N) &\lesssim_u  N^{3/2} K^{1/2} \quad \text{whenever}\quad N\geq \Biggl( \frac{\int_I N(t)^2\, dt}{\int_I N(t)^{-1}\, dt} \Biggr)^{1/3}
\end{align}
and so, in particular, when $N\geq N_{max}:= \sup_{t\in I}N(t)$.  This is the base step for the inductive proof of Theorem~\ref{T:LTS}.
The passage to smaller values of $N$ relies on the following:

\begin{lemma}[Recurrence relations for $A(N)$ and $\tilde A_q(N)$]\label{L:A rec}
For $\eta$ sufficiently small,
\begin{align}
A(N)&\lesssim_u 1  + c^{-3/2}N^{3/2}K^{1/2} + \eta^2 \tilde A_q(2N)\label{E:A rec}\\
\tilde A_q(N)&\lesssim_u 1  + c^{-3/2}N^{3/2}K^{1/2} + \eta A(N) + \eta^2 \tilde A_q(2N),\label{E:tilde A rec}
\end{align}
uniformly in $N\in 2^\Z$.  Here $c=c(\eta)$ as in \eqref{c_0}.
\end{lemma}

\begin{proof}
The recurrence relations for $A(N)$ and $\tilde A_q(N)$ rely on Lemma~\ref{L:Strichartz} and Proposition~\ref{P:MaxS}, respectively.
To estimate the contribution of the nonlinearity, we decompose $u(t)=u_{\leq c N(t)}(t) + u_{>cN(t)}(t)$ and then selectively $u=u_{\leq N} + u_{>N}$.  Recalling that the $\Oh$ notation
incorporates possible additional Littlewood--Paley projections, we may write
\begin{align}
F(u)&= \Oh\bigl(  u_{>c N(t)}^2 u^3 \bigr) + \Oh\bigl( u_{\leq cN(t)}^2 u^3 \bigr) \notag\\
&= \Oh\bigl(  u_{>c N(t)}^2 u^3 \bigr)  + \Oh\bigl(  u_{\leq c N(t)}^2 u_{>N}^2 u \bigr) + \Oh\bigl(  u_{\leq c N(t)}^2 u_{\leq N}^2 u \bigr).  \label{decomp}
\end{align}
Using this decomposition together with Lemma~\ref{L:Strichartz} and Bernstein's inequality, we obtain
\begin{align}\label{A est}
A(N)&\lesssim \|\nabla u_{\leq N}\|_{L_t^\infty L_x^2} + \bigl\|\nabla P_{\leq N} \Oh\bigl(u_{>c N(t)}^2 u^3 \bigr)\bigr\|_{L_t^2L_x^{6/5}}\notag\\
&\quad+\bigl\|\nabla P_{\leq N} \Oh\bigl( u_{\leq c N(t)}^2 u_{>N}^2 u \bigr)\bigr\|_{L_t^2L_x^{6/5}}+\bigl\|\nabla P_{\leq N} \Oh\bigl(u_{\leq c N(t)}^2u_{\leq N}^2 u \bigr)\bigr\|_{L_t^2L_x^{6/5}}\notag\\
&\lesssim_u 1 +  N^{3/2}\| u_{>c N(t)}^2 u^3 \|_{L_t^2L_x^1} +N^{3/2}\| u_{\leq c N(t)}^2 u_{>N}^2 u\|_{L_t^2L_x^1}\notag\\
&\quad +\bigl\|\nabla \Oh\bigl(u_{\leq c N(t)}^2u_{\leq N}^2 u \bigr)\bigr\|_{L_t^2L_x^{6/5}}.
\end{align}
Using instead Proposition~\ref{P:MaxS} and Bernstein's inequality, we find
\begin{align}\label{tilde A est}
\tilde A_q(N) &\lesssim N^{3/2}\Bigl\{ \bigl\| |\nabla|^{-1/2} u_{\geq N}\bigr\|_{L_t^\infty L_x^2} + \| u_{>c N(t)}^2 u^3 \|_{L_t^2L_x^1} + \| u_{\leq c N(t)}^2 u_{>N}^2 u\|_{L_t^2L_x^1}\notag\\
&\qquad \qquad + \bigl\| |\nabla|^{-1/2} P_{\geq N} \Oh\bigl(u_{\leq c N(t)}^2u_{\leq N}^2 u \bigr) \bigr\|_{L_t^2 L_x^{6/5}} \Bigr\}\notag\\
&\lesssim_u 1 +  N^{3/2}\| u_{>c N(t)}^2 u^3 \|_{L_t^2L_x^1} +N^{3/2}\| u_{\leq c N(t)}^2 u_{>N}^2 u\|_{L_t^2L_x^1}\notag\\
&\quad +\bigl\|\nabla \Oh\bigl(u_{\leq c N(t)}^2u_{\leq N}^2 u \bigr)\bigr\|_{L_t^2L_x^{6/5}}.
\end{align}
Therefore, to obtain the desired recurrence relations it remains to estimate the (identical) last three terms on the right-hand sides of \eqref{A est} and \eqref{tilde A est}.  We will consider
these terms individually, working from left to right.

To treat the first term, we decompose the time interval $I$ into characteristic subintervals $J_k$ where $N(t)\equiv N_k$.  On each of these subintervals, we
apply H\"older's inequality, Sobolev embedding, Bernstein's inequality, and Lemma~\ref{L:ST-N(t)} to obtain
\begin{align*}
\|u_{>c N(t)}^2 u^3 \|_{L_t^2L_x^1(J_k\times\R^3)}
&\lesssim \|u_{>c N_k}\|_{L_{t,x}^4(J_k\times\R^3)}^2 \|u \|_{L_t^\infty L_x^6}^3 \\
&\lesssim_u c^{-3/2} N_k^{-3/2} \|\nabla u_{>c N_k}\|_{L_t^4 L_x^3(J_k\times\R^3)}^2\\
&\lesssim_u c^{-3/2} N_k^{-3/2}.
\end{align*}
Squaring and summing the estimates above over the subintervals $J_k$, we find
\begin{align}\label{1}
N^{3/2} \bigl\| u_{>c N(t)}^2 u^3 \bigr\|_{L_t^2 L_x^1}\lesssim_u c^{-3/2} N^{3/2} K^{1/2},
\end{align}
which is the origin of this term on the right-hand sides of \eqref{E:A rec} and \eqref{E:tilde A rec}.

To estimate the second term, we begin with a preliminary computation:  Using Bernstein's inequality and Schur's test (for the last step), we estimate
\begin{align}
\bigl\| \Oh&\bigl( u_{>N}^2 u \bigr) \bigr\|_{L^2_tL^{3/2}_x} \notag\\
&\lesssim \biggl\|\sum_{\substack{M_1\geq M_2\geq M_3\\[0.2ex] M_2 > N}}
    \| u_{M_1}(t) \|_{L^2_x} \| u_{M_2}(t) \|_{L^q_x} \| u_{M_3}(t) \|_{L^{\frac{6q}{q-6}}_x}\biggr\|_{L_t^2}   \notag \\[-1ex]
&\lesssim \biggl\| \sup_{M> N} \|M^{\frac3q-1} u_{M}(t) \|_{L^q_x} \sum_{M_1\geq M_3} \bigl(\tfrac{M_3}{M_1}\bigr)^{3/q}
    \| \nabla u_{M_1}(t) \|_{L^2_x} \| \nabla u_{M_3} (t) \|_{L^2_x} \biggr\|_{L_t^2}  \notag \\
&\lesssim_u N^{-3/2}\tilde A_q(2N) .  \label{paratripleprod}
\end{align}
Using this, H\"older, and \eqref{c_0}, we find
\begin{align}
N^{3/2} \| u_{\leq c N(t)}^2 u_{>N}^2 u \|_{L_t^2 L_x^1}
&\lesssim N^{3/2} \| u_{\leq c N(t)} \|_{L_t^\infty L_x^6}^2 \| \Oh\bigl( u_{>N}^2 u \bigr) \|_{L^2_t L^{3/2}_x}\notag\\
&\lesssim_u \eta^2 \tilde A_q(2N) \label{6}.
\end{align}
This is the origin of the last term on the right-hand sides of \eqref{E:A rec} and \eqref{E:tilde A rec}.

Finally, to estimate the contribution coming from the last term in \eqref{A est} and \eqref{tilde A est}, we distribute the gradient, use H\"older's inequality,
and then \eqref{c_0} and \eqref{duh}:
\begin{align}\label{4}
\bigl\| \nabla \Oh\bigl(  u_{\leq c N(t)}^2  u_{\leq N}^2 u \bigr) \bigr\|_{L_t^2 L_x^{6/5}}
&\lesssim \|\nabla u_{\leq N}\|_{L_t^2 L_x^6} \|u_{\leq c N(t)}\|_{L_t^\infty L_x^6} \|u \|_{L_t^\infty L_x^6}^3 \notag\\
    &\quad + \|\nabla u \|_{L_t^\infty L_x^2} \|u_{\leq c N(t)}\|_{L_t^\infty L_x^6} \|u_{\leq N}\|_{L_t^4 L_x^\infty}^2 \| u \|_{L_t^\infty L_x^6} \notag \\
&\lesssim_u \eta A(N).
\end{align}
As $A(N)$ is known to be finite (cf. \eqref{E:A finite}), this can be brought to the other side of \eqref{E:A rec}; naturally, this requires $\eta$ to be sufficiently small
depending on $u$ and certain absolute constants, but not on $I$.

Collecting estimates \eqref{1} through \eqref{6} and choosing $\eta$ sufficiently small, this completes the proof of the lemma.
\end{proof}

We now have all the ingredients needed to complete the proof of Theorem~\ref{T:LTS}.

\begin{proof}[Proof of Theorem~\ref{T:LTS}]
With the base step \eqref{E:A:N large} and Lemma~\ref{L:A rec} in place, Theorem~\ref{T:LTS} follows from a straightforward induction argument, provided $\eta$ is chosen sufficiently small depending on $u$.
\end{proof}

\begin{remark}\label{R:no better}
In the introduction it was asserted that the long-time Strichartz estimates in Theorem~\ref{T:LTS} are essentially best possible in the focusing case. We now
elaborate that point.  For the energy-critical equation, the principal difficulty is to obtain control over the low frequencies, because all known conservation
laws (with the exception of energy) and monotonicity formulae are energy-subcritical.  If (by some miracle) we knew our putative minimal counterexample $u$ belonged
to $L^\infty_t L^2_x$, the whole argument could be brought to a swift conclusion, even in the focusing case (cf. \cite{Berbec}).  Thus any potential improvement of
Theorem~\ref{T:LTS} should be judged by whether it gives better control on the low frequencies.

It is well-known that
\begin{equation}\label{Wdefn}
W(x) = \bigl(1+\tfrac13|x|^2)^{-1/2} \quad\text{obeys}\quad \Delta W + W^5 =0
\end{equation}
and so is a static solution of the focusing energy-critical NLS.  In particular, it is almost periodic with parameters $N(t)\equiv1$ and $x(t)\equiv0$.

As $\int W(x)^5 \, dx = 4\pi\sqrt{3}$, we can read off from \eqref{Wdefn} that
\begin{equation}\label{W hat}
\hat W(\xi) = 4\pi\sqrt{3} |\xi|^{-2} + O\bigl(|\xi|^\eps\bigr) \quad\text{as} \quad \xi\to 0
\end{equation}
and so deduce $\| W_M \|_{L^q} \sim M^{1-3/q}$ for $M$ small and $6\leq q \leq\infty$.  This shows that the supremum is essential in \eqref{tilde A}; we cannot
expect the bound \eqref{E:A bound} for the sum of the Littlewood--Paley pieces.  It also shows that the $L^2_tL^6_x$ norm of $\nabla W_{\leq N}$ on long time intervals
decays no faster than the $N^{3/2}$ rate proved for $A(N)$.
\end{remark}

%
%
%
%

\section{Impossibility of rapid frequency cascades}\label{S:cascade}

In this section, we show that the first type of almost periodic solution described in Theorem~\ref{T:enemies} (for
which $\int_0^{\Tmax}N(t)^{-1}\,dt<\infty$) cannot exist.  We will show that its existence is inconsistent with the
conservation of mass, $M(u):=\int_{\R^3} |u(t,x)|^2\, dx$.  The argument does not utilize the defocusing nature of the equation beyond the fact that the
solution belongs to $L^\infty_t \dot H^1_x$.

\begin{lemma}[Finite mass]\label{L:mass}
Let $u:[0, \Tmax)\times\R^3\to \C$ be an almost periodic solution to \eqref{nls} with  $\|u\|_{L^{10}_{t,x}( [0, \Tmax) \times \R^3)} =+\infty$ and
\begin{align}\label{finite K}
K:=\int_0^{\Tmax} N(t)^{-1}\,dt<\infty.
\end{align}
{\rm(}Note $\Tmax=\infty$ is allowed.{\rm)}  Then $u\in L^\infty_t L^2_x$; indeed, for all $0<N<1$,
\begin{align}\label{D}
\|u_{N\leq \cdot\leq 1}\|_{L_t^\infty L_x^2([0, \Tmax)\times\R^3)} + \tfrac1N\Bigl\{ \sum_{M<N} \|\nabla u_{M}\|_{L_t^2 L_x^6([0, \Tmax)\times\R^3)}^2\Bigr\}^{1/2}
    &\lesssim_u 1.
\end{align}
\end{lemma}

\begin{proof}
The key point is to prove \eqref{D}; finiteness of the mass follows easily from this.  Indeed, letting $N\to 0$ in \eqref{D} to control the low frequencies and
using $\nabla u\in L^\infty_tL^2_x$ and Bernstein for the high frequencies, we obtain
\begin{align}\label{finite mass}
\|u\|_{L_t^\infty L_x^2}\leq \|u_{\leq 1}\|_{L_t^\infty L_x^2} +\|u_{>1}\|_{L_t^\infty L_x^2}\lesssim_u 1.
\end{align}
In the inequality above and for the remainder of the proof all spacetime norms are over $[0,\Tmax)\times\R^3$.

As $K$ is finite, the conclusion \eqref{E:A bound} of Theorem~\ref{T:LTS} extends (by exhaustion) to the time interval $[0,\Tmax)$.   Observe that the second summand
in \eqref{D} is $N^{-1} A(N/2)$, in the notation of that theorem.

We will estimate the left-hand side of \eqref{D} by a small multiple of itself plus a constant.  For this statement to be meaningful, we need the left-hand side of \eqref{D} to be finite.
This follows easily from Theorem~\ref{T:LTS} and Bernstein's inequality:
\begin{align}
\text{LHS}\eqref{D}\lesssim N^{-1} \|\nabla u\|_{L_t^\infty L_x^2} + N^{-1} A(N/2) \lesssim_u N^{-1} (1 +N^3K)^{1/2}<\infty.\label{LHS finite}
\end{align}

The origin of the small constant lies with the almost periodicity of the solution.  Indeed, by Remark~\ref{R:small freq} and Sobolev embedding, for $\eta>0$
(a small parameter to be chosen later)  there exists $c=c(\eta)$ such that
\begin{align}\label{c}
\| u_{\leq c N(t)}\|_{L_t^\infty L^6_x} + \|\nabla u_{\leq c N(t)}\|_{L_t^\infty L^2_x}\leq \eta.
\end{align}

To continue, fix $0<N<1$. Using the Duhamel formula from Proposition~\ref{P:duhamel} together with the Strichartz inequality we obtain
\begin{align}
\text{LHS}\eqref{D}\lesssim \tfrac1N \|\nabla P_{<N} F(u)\|_{L_t^2L_x^{6/5}} + \|P_{N\leq \cdot\leq 1} F(u)\|_{L_t^2L_x^{6/5}}. \label{D est}
\end{align}
To estimate the nonlinearity, we decompose $u(t)=u_{\leq c N(t)}(t) + u_{>cN(t)}(t)$ and then $u=u_{<N} + u_{N\leq \cdot \leq 1} + u_{>1}$.  As the
$\Oh$ notation incorporates possible additional Littlewood--Paley projections, we may write
\begin{align}
F(u)&= \Oh\bigl(  u_{>c N(t)}^2 u^3 \bigr)  + \Oh\bigl(  u_{\leq c N(t)} u_{< N}^2 u^2 \bigr) + \Oh\bigl(  u_{\leq c N(t)} u_{\leq 1}^2  u_{N\leq\cdot\leq 1} u \bigr) \notag\\
&\quad+\Oh\bigl(  u_{\leq c N(t)} u_{>1}^2 u^2 \bigr).  \label{decomposition}
\end{align}
Next, we estimate the contributions of each of these terms to \eqref{D est}, working from left to right.

Using Bernstein's inequality and \eqref{1}, we bound the contribution of the first term as follows:
\begin{align*}
\tfrac1N \bigl\|\nabla P_{<N} \Oh\bigl(  u_{>c N(t)}^2 u^3 \bigr)&\bigr\|_{L_t^2L_x^{6/5}}
        + \bigl\|P_{N\leq \cdot\leq 1} \Oh\bigl(  u_{>c N(t)}^2 u^3 \bigr)\bigr\|_{L_t^2L_x^{6/5}}\\
&\lesssim (N^{1/2} + 1) \|\Oh(u_{>c N(t)}^2 u^3)\|_{L_t^2L_x^1} \\
&\lesssim_u c^{-3/2} K^{1/2}.
\end{align*}

To estimate the contribution of the second term in \eqref{decomposition} to \eqref{D est}, we use Bernstein's inequality on the second summand and distribute the gradient,
followed by H\"older's inequality, \eqref{duh}, and \eqref{c}:
\begin{align*}
\tfrac1N \bigl\|\nabla P_{<N} &\Oh\bigl(  u_{\leq c N(t)} u_{<N}^2 u^2 \bigr)\bigr\|_{L_t^2L_x^{6/5}}
        + \bigl\|P_{N\leq \cdot\leq 1} \Oh\bigl(  u_{\leq c N(t)} u_{<N}^2 u^2 \bigr)\bigr\|_{L_t^2L_x^{6/5}}\\
&\lesssim \tfrac1N \|\nabla u_{\leq c N(t)}\|_{L_t^\infty L_x^2} \|u_{<N}\|_{L_t^4 L_x^\infty}^2 \| u \|_{L_t^\infty L_x^6}^2 \\
&\quad + \tfrac1N  \|u_{\leq c N(t)}\|_{L_t^\infty L_x^6}  \|\nabla u_{<N}\|_{L_t^2 L_x^6} \|u \|_{L_t^\infty L_x^6}^3 \\
&\quad + \tfrac1N \|u_{\leq c N(t)}\|_{L_t^\infty L_x^6} \|u_{<N}\|_{L_t^4 L_x^\infty}^2  \|\nabla u \|_{L_t^\infty L_x^2} \| u \|_{L_t^\infty L_x^6} \notag \\
&\lesssim_u \eta \, \text{LHS}\eqref{D}.
\end{align*}

Using Bernstein's inequality, Theorem~\ref{T:LTS}, \eqref{duh}, and \eqref{c}, we estimate the contribution of the third term in \eqref{decomposition} as follows:
\begin{align*}
\tfrac1N \bigl\|\nabla P_{<N} &\Oh\bigl(  u_{\leq c N(t)} u_{\leq 1}^2  u_{N\leq\cdot\leq 1} u \bigr)\bigr\|_{L_t^2L_x^{6/5}}
        + \bigl\| P_{N\leq\cdot\leq 1}\Oh\bigl(  u_{\leq c N(t)} u_{\leq 1}^2  u_{N\leq\cdot\leq 1} u \bigr)\bigr\|_{L_t^2L_x^{6/5}}\\
&\lesssim \|u_{\leq c N(t)}\|_{L_t^\infty L_x^6} \|u_{\leq 1}\|_{L_t^4 L_x^\infty}^2\|u_{N\leq\cdot\leq 1}\|_{L_t^\infty L_x^2} \|u\|_{L_t^\infty L_x^6}\\
&\lesssim_u \eta (1+K^{1/2}) \, \text{LHS}\eqref{D}.
\end{align*}

Finally, to estimate the contribution to \eqref{D est} of the last term in \eqref{decomposition} we use Bernstein's inequality, Theorem~\ref{T:LTS}, \eqref{paratripleprod}, and \eqref{c}:
\begin{align*}
\tfrac1N \bigl\|\nabla P_{\leq N/2} &\Oh\bigl(  u_{\leq c N(t)} u_{>1}^2 u^2 \bigr)\bigr\|_{L_t^2L_x^{6/5}}
        + \bigl\|P_{N\leq \cdot\leq 1} \Oh\bigl(  u_{\leq c N(t)} u_{>1}^2 u^2 \bigr)\bigr\|_{L_t^2L_x^{6/5}}\\
&\lesssim (N^{1/2}+1) \bigl\|\Oh\bigl( u_{\leq c N(t)} u_{>1}^2 u^2 \bigr)\bigr\|_{L_t^2L_x^1}\\
&\lesssim  \|u_{\leq c N(t)}\|_{L_t^\infty L_x^6} \|\Oh(u_{>1}^2 u)\|_{L_t^2L_x^{3/2}} \|u\|_{L_t^\infty L_x^6} \\
&\lesssim_u \eta (1+K^{1/2}).
\end{align*}

Collecting all the estimates above, \eqref{D est} implies
$$
\text{LHS}\eqref{D}\lesssim_u \eta (1+K^{1/2}) \, \text{LHS}\eqref{D} + 1+c^{-3/2}K^{1/2}.
$$
Recalling \eqref{finite K} and \eqref{LHS finite} and taking $\eta$ small enough depending on $u$ and $K$ yields \eqref{D}.
\end{proof}

We are now ready to prove the main result of this section:

\begin{theorem}[No rapid frequency-cascades]\label{T:no cascade}
There are no almost periodic solutions $u:[0, \Tmax)\times\R^3\to \C$ to \eqref{nls} with  $\|u\|_{L^{10}_{t,x}( [0, \Tmax) \times \R^3)} =+\infty$ and
\begin{align}\label{finite int}
\int_0^{\Tmax} N(t)^{-1}\,dt<\infty.
\end{align}
\end{theorem}

\begin{proof}
We argue by contradiction.  Let $u$ be such a solution.  By Corollary~\ref{C:blowup criterion},
\begin{align}\label{N blowup}
\lim_{t\to \Tmax} N(t)=\infty,
\end{align}
when $\Tmax$ is finite; this is also true when $\Tmax$ is infinite by virtue of \eqref{finite int}.

We will prove that the existence of such a solution $u$ is inconsistent with the conservation of mass. In Lemma~\ref{L:mass} we found
that the mass is finite; to derive the desired contradiction we will prove that the mass is not only finite, but zero!

We first show that the mass at low frequencies is small.
To do this, we use the Duhamel formula from Proposition~\ref{P:duhamel} together with the Strichartz inequality, followed by Bernstein's inequality:
\begin{align*}
\|u_{\leq N}\|_{L_t^\infty L_x^2}&\lesssim \|P_{\leq N} F(u)\|_{L_t^2 L_x^{6/5}}\lesssim N^{1/2}\|F(u)\|_{L_t^2 L_x^1}.
\end{align*}
In the display above and for the remainder of the proof all spacetime norms are over $[0,\Tmax)\times\R^3$.

To estimate the nonlinearity we decompose it as follows:
$$
F(u)=\Oh(u_{\leq 1}^3 u^2) + \Oh(u_{>1}^3 u^2).
$$
By Theorem~\ref{T:LTS}, \eqref{duh}, \eqref{finite int}, Bernstein, and finiteness of the mass,
\begin{align*}
\|\Oh(u_{\leq 1}^3 u^2) \|_{L_t^2 L_x^1}
\lesssim \|u_{\leq 1}\|^2_{L_t^4L_x^\infty}  \|u_{\leq 1}\|_{L_{t,x}^\infty} \|u\|_{L_t^\infty L_x^2}^2\lesssim_u 1,
\end{align*}
while by Theorem~\ref{T:LTS}, \eqref{paratripleprod}, and \eqref{finite int},
\begin{align*}
\|\Oh(u_{>1}^3 u^2)\|_{L_t^2 L_x^1}
\lesssim \|u\|_{L_t^\infty L_x^6}^2\|u_{>1}^3\|_{L_t^2L_x^{3/2}} \lesssim_u 1.
\end{align*}
Thus,
$$
\|u_{\leq N}\|_{L_t^\infty L_x^2}\lesssim_u N^{1/2}.
$$

By comparison, control over the mass at middle and high frequencies can be obtained with just Bernstein's inequality and the fact that
for any $\eta>0$ there exists $c=c(u,\eta)>0$ so that
$$
\| \nabla u_{\leq c N(t)}(t) \|_{L_x^2} \leq \eta,
$$
which was noted in Remark~\ref{R:small freq}.  Altogether, we have that for any $t\in [0, \Tmax)$,
\begin{align*}
\| u(t) \|_{L_x^2} & \lesssim \| u_{\leq N}(t) \|_{L_x^2} + \| P_{>N} u_{\leq c N(t)}(t) \|_{L_x^2} + \| u_{> c N(t)}(t) \|_{L_x^2} \\
&\lesssim_u N^{1/2} + N^{-1} \| \nabla u_{\leq c N(t)}(t) \|_{L_x^2} + c^{-1} N(t)^{-1} \| \nabla u \|_{L^\infty_t L_x^2} \\
&\lesssim_u N^{1/2} + N^{-1} \eta + c^{-1} N(t)^{-1}.
\end{align*}
Using \eqref{N blowup}, we can make the right-hand side here as small as we wish.  (Choose $N$ small, then $\eta$ small, and then $t$ close to $\Tmax$.)
Because mass is conserved under the flow, this allows us to conclude that $\|u\|_{L^\infty_t L^2_x}=0$ and thus $u\equiv0$ in contradiction to the
hypothesis $\|u\|_{L^{10}_{t,x}( [0, \Tmax) \times \R^3)} =+\infty$.
\end{proof}

%
%
%
%

\section{The frequency-localized interaction Morawetz inequality}\label{S:IM}

In this section, we prove a spacetime bound on the high-frequency portion of the solution:

\begin{theorem}[A frequency-localized interaction Morawetz estimate] \label{T:flim}
Suppose $u:[0, \Tmax)\times\R^3\to \C$ is an almost periodic solution to \eqref{nls} such that $N(t)\geq 1$ and let $I\subset [0, \Tmax)$ be a union of
contiguous characteristic intervals $J_k$.  Fix $0<\eta_0\leq1$.  For $N>0$ sufficiently small (depending on $\eta_0$ but not on $I$),
\begin{align}\label{E:B bound}
  \int_I\int_{\R^{3}} |u_{> N}(t,x)|^4\,dx\,dt&\lesssim_u \eta_0 \bigl( N^{-3} + K \bigr),
\end{align}
where $K:=\int_I N(t)^{-1}\,dt$.  Importantly, the implicit constant in the inequality above does not depend on $\eta_0$ or the interval $I$.
\end{theorem}

Unlike Theorem~\ref{T:LTS}, the argument does not rely solely on estimates for the linear propagator and is not indifferent to the sign of the nonlinearity.
Instead, we use a special monotonicity formula associated with \eqref{nls}, namely, the interaction Morawetz identity.  This is a modification of the
traditional Morawetz identity (cf. \cite{LinStrauss,Morawetz75}) introduced in \cite{CKSTT:interact}.  We begin with a general form of the identity:

\begin{proposition}\label{P:Interact}
Suppose $i\partial_t \phi = -\Delta \phi + |\phi|^4\phi + \ef$ and let
\begin{equation}\label{E:IM defn}
M(t):= 2 \iint_{\R^3\times\R^3} |\phi(y)|^2 a_k(x-y) \Im \{  \phi_k(x) \bar \phi(x) \} \,dx\,dy,
\end{equation}
for some weight $a:\R^d\to\R$.  Then
\begin{align}
\partial_t M(t) = \int_{\R^3}\!\!\int_{\R^3} \Bigl\{ & \tfrac{4}{3} a_{kk}(x-y) |\phi(x)|^{6}|\phi(y)|^2 \label{GIM:good} \\
  &\hspace*{-1.5em}{} + 2a_k(x-y) |\phi(y)|^2 \Re\bigl[\phi_k(x)\bar \ef(x) - \ef_k(x)\bar \phi(x)\bigr]  \label{GIM:pb} \\
  &\hspace*{-3.0em}{} + 4a_k(x-y) (\Im \ef(y) \bar \phi(y)) (\Im  \phi_k(x) \bar \phi(x)) \label{GIM:mb} \\
  &\hspace*{-4.5em}{} + 4a_{jk}(x-y) \bigl[  |\phi(y)|^{2} \bar \phi_{j}(x) \phi_{k}(x) - (\Im \bar \phi(y) \phi_j(y))(\Im \bar \phi(x) \phi_k(x)) \bigr] \label{GIM:grad} \\
  &\hspace*{-6.0em}{} - a_{jjkk}(x-y) \, |\phi(y)|^{2} |\phi(x)|^{2} \Bigr\} \, dx \,dy. \label{GIM:jjkk}
\end{align}
Subscripts denote spatial derivatives and repeated indices are summed.
\end{proposition}

The significance of this identity to our problem is best seen by choosing $a(x)=|x|$ and $\phi$ to be a solution to \eqref{nls}.
In this case, $\ef=0$ and the Fundamental Theorem of Calculus yields
$$
8\pi\!\! \int_I  \int_{\R^3}|\phi(t,x)|^4\, dx\,dt \leq 2 \|M(t)\|_{L_{t}^{\infty}(I)}
    \leq 4\|\phi\|_{L_{t}^{\infty}L_{x}^{2}(I\times\R^{3})}^{3} \|\phi\|_{L_{t}^{\infty}\dot{H}_{x}^{1}(I\times\R^{3})}.
$$
The left-hand side originates from \eqref{GIM:jjkk}; the terms \eqref{GIM:grad} and \eqref{GIM:good} are both positive.

Unfortunately for us, a minimal blowup solution need not have finite $L^2_x$ norm at any time.  Thus it is necessary to localize the identity to high frequencies, that is,
choose $\phi=u_{>N}$.  Naturally,
this produces myriad error terms; nevertheless, in spatial dimensions four and higher they can be controlled (cf. \cite{RV,thesis:art,Visan:4D}).  In the three dimensional case
under consideration here, there is one error term (originating from \eqref{GIM:mb}) that cannot be satisfactorily controlled.  (See also Remark~\ref{R:dream} at the end of this section.)
This was observed already in \cite{CKSTT:gwp} and as there, our solution is to truncate the function $a$.  This truncation ruins the convexity properties of $a$ that made some of the terms in
Proposition~\ref{P:Interact} positive, thus creating more error terms to control.

For reasons we will explain in due course, it is important to perform the cutoff of $a$ in a very careful fashion.  We choose $a$ to be a smooth spherically symmetric function,
which we regard interchangeably as a function of $x\in\R^3$ or $r=|x|$.  We specify it further in terms of its radial derivative:
\begin{equation}\label{a Defn}
\!\!\! a(0)=0, \  a_r \geq 0, \ a_{rr}\leq 0, \ \text{and}\  a_r = \begin{cases}
1                 &: r\leq R \\
1-J^{-1}\log(r/R)\!\! &: eR \leq r \leq e^{J-J_0} R \\
0                 &: e^{J} R\leq r
\end{cases}
\end{equation}
where $J_0\geq 1$, $J\geq 2J_0$, and $R$ are parameters that will be determined in due course.  It is not difficult to see that one may fill in the regions where $a_r$ is not yet defined so
that the function obeys
\begin{equation}\label{a symbolic}
|\partial_r^k a_r| \lesssim_k  J^{-1} r^{-k} \qquad\text{for each $k\geq 1$,}
\end{equation}
uniformly in $r$ and in the choice of parameters.

When $|x|\leq R$, we see that $a(x)=|x|$, while $a$ is a constant when $|x|\geq e^JR$.  The key point about the transition between these two regimes is that
\begin{equation}\label{Lap a}
 \frac2r a_r \geq \frac{2J_0}{Jr} \quad\text{but}\quad |a_{rr}|  \leq \frac{1}{Jr}
\end{equation}
when $eR \leq r \leq e^{J-J_0} R$.  Thus the Laplacian $a_{kk} = a_{rr}+\frac2r a_r$ is dominated by the first derivative term and so remains coercive at these radii.
(This also appears implicitly in \cite[\S11]{CKSTT:gwp} and is the key point behind the `averaging over $R$' argument there.)

As noted above, we will be applying Proposition~\ref{P:Interact} with
\begin{align}
\phi = \uhi := u_{> N}, \quad\text{and so}\quad \ef = \ph F(u) - F(\uhi).
\end{align}
(We will also write $\ulo:=u_{\leq N}$.)  Here $N$ is an additional parameter that will be chosen small (depending on $\eta_0$ and $u$).  We require that $N$, $R$, and $J$ are related via
\begin{align}\label{E:JRN}
e^J R N =1.
\end{align}
Actually, it is merely essential that $e^JRN\leq 1$, but choosing equality makes the exposition simpler.  Our first restriction on these parameters is that
$N$ is small enough and $R$ is large enough so that given $\eta=\eta(\eta_0,u)$,
\begin{equation}\label{E:etas job}
\int_{\R^3} |\nabla \ulo(t,x)|^2 \,dx + \int_{\R^3} |N \uhi(t,x)|^2 \,dx  + \int_{|x-x(t)|>\frac R2} |\nabla \uhi(t,x)|^2\,dx < \eta^2
\end{equation}
uniformly for $0\leq t<\Tmax$.  The possibility of doing this follows immediately from the fact that $u$ is almost periodic modulo symmetries and $N(t)\geq 1$.

Before moving on to estimating the terms in Proposition~\ref{P:Interact}, we pause to review the tools at our disposal.  Besides using the norm $\|\uhi\|_{L^4_{t,x}}$ to estimate itself, we
will also make recourse to Theorem~\ref{T:LTS} and Proposition~\ref{P:DDS}.  For ease of reference, we record these results in the forms we will use:

\begin{corollary}[A priori bounds]\label{C:all bounds}  For all $\tfrac2q+\tfrac3r=\tfrac32$ with $2\leq q\leq \infty$ and any $s<1-\frac3q$,
\begin{align}
\bigl\|\nabla \ulo \bigr\|_{L_t^q L_x^r} + \bigl\| N^{1-s} |\nabla|^s \uhi \bigr\|_{L_t^q L_x^r} \lesssim_u \bigl( 1 + N^3K \bigr)^{1/q}. \label{grads in Strich}
\end{align}
Under the hypothesis \eqref{E:etas job},
\begin{align}
\|\ulo\|_{L_t^4 L_x^\infty} &\lesssim_u \eta^{1/2} \bigl( 1 + N^{3}K \bigr)^{1/4}. \label{ulo in Strich}
\end{align}
Furthermore, for any $\rho\leq Re^J =N^{-1}$,
\begin{align}\label{E:22 bound}
\int_I \; \sup_{x\in\R^3} \; \int_{|x-y|\leq \rho} |\uhi(t,y)|^2 \,dy\,dt  \lesssim_u \rho \bigl( K + N^{-3} \bigr).
\end{align}
\end{corollary}

\begin{proof}
Recall that Theorem~\ref{T:LTS} implies
$$
A(M):=\Bigl\{ \sum_{M'\leq M} \|\nabla u_{M'}\|_{L_t^2 L_x^6(I\times\R^3)}^2\Bigr\}^{1/2} \lesssim_u  (1 + M^{3} K)^{1/2}
$$
uniformly in $M$.  Setting $M=N$ yields all the estimates on $\ulo$ stated in the corollary.  More explicitly, the $q=2$ case of \eqref{grads in Strich} as well as \eqref{ulo in Strich}
follow from this statement and \eqref{duh}. The other values of $q$ can then be deduced by interpolation with the (conserved) energy.

Similarly, to estimate $\uhi$ we write
$$
M^{1-s} \| |\nabla|^s u_M \|_{L_t^q L_x^r} \lesssim \| \nabla u_M \|_{L_t^q L_x^r} \lesssim A(M)^{2/q} \|\nabla u\|_{L^\infty_tL^2_x}^{(q-2)/q} \lesssim_u (1 + M^{3} K)^{1/q},
$$
multiply through by $M^{s-1}$, and sum over $M\geq N$.  Notice that the condition $\frac3q+s <1$ guarantees the convergence of this sum.

Claim \eqref{E:22 bound} will follow by combining Proposition~\ref{P:DDS} and Theorem~\ref{T:LTS}.  First we write $(i\partial_t +\Delta)u_{>N}=F+G$ with
$F=P_{>N}\Oh(u_{>N}^2 u^3)$ and $G=P_{>N}\Oh(u_{\leq N}^4 u)$ and then estimate these as follows:  By Theorem~\ref{T:LTS} and \eqref{paratripleprod},
\begin{align*}
\| F \|_{L^2_t L^1_x} &\lesssim \|u\|_{L^\infty_t L^6_x}^2 \| \Oh(u_{>N}^2 u) \|_{L^2_t L^{3/2}_x} \lesssim_u N^{-3/2} + K^{1/2} ,
\end{align*}
while by Bernstein, Theorem~\ref{T:LTS}, and \eqref{duh},
\begin{align*}
\| G \|_{L^2_t L^{6/5}_x} &\lesssim N^{-1} \| \nabla \Oh(u_{\leq N}^4 u) \|_{L^2_t L^{6/5}_x}
    \lesssim N^{-1} \|\nabla u\|_{L^\infty_t L^2_x} \| u \|_{L^\infty_t L^6_x}^2 \| u_{\leq N} \|_{L^4_t L^\infty_x}^2\\
&\lesssim_u N^{-1} +N^{1/2}K^{1/2}.
\end{align*}
Putting these together with Proposition~\ref{P:DDS} yields
\begin{align*}
\rho^{1/2} \bigl\| \smudge u_{>N} \bigr\|_{L^2_t L^\infty_x(I\times\R^3)} \lesssim_u N^{-1} + \bigl(N^{1/2} + \rho^{-1/2}\bigr)\bigl( K + N^{-3} \bigr)^{1/2}.
\end{align*}
Noting from \eqref{E:smudge} that, modulo a factor of $\rho^{-3/2}$, $\smudge u_{>N}(t,x)$ controls the $L^2_x$ norm on the ball around $x$, and recalling the restriction on $\rho$, we deduce the claim.
\end{proof}

We now begin our analysis of the individual terms in Proposition~\ref{P:Interact}, beginning with the most important one:

\begin{lemma}[Mass-mass interactions]
\begin{align*}
 8\pi \|\uhi\|_{L^4_{t,x}(I\times\R^3)}^4 - \int_I\!\iint {-a_{jjkk}}(x-y) &\, |\uhi(y)|^{2} |\uhi(x)|^{2} \, dx \,dy\,dt  \\
 &\lesssim_u\frac{ \eta^2e^{2J}}{J} \bigl( K + N^{-3} \bigr).
\end{align*}
\end{lemma}

\begin{proof}
In three dimensions, $\Delta |x| = 2|x|^{-1}$ and $-(4\pi|x|)^{-1}$ is the fundamental solution of Laplace's equation.  In this way, we are left to estimate the error terms
originating from the truncation of $a$ at radii $|x-y|\geq R$.  Combining \eqref{a symbolic} and \eqref{E:22 bound} yields
\begin{align*}
\int_I \iint_{|x-y|\geq R} & \bigl|a_{jjkk}(x-y)\bigr| \, |\uhi(y)|^{2} |\uhi(x)|^{2} \, dx \,dy\,dt \\
&\lesssim_u J^{-1} \|\uhi\|_{L^\infty_t L^2_x}^2 \sum_{j=0}^{J} (Re^j)^{-3} (Re^j) \bigl( K + N^{-3} \bigr).
\end{align*}
To obtain the lemma, we simply invoke \eqref{E:etas job} as well as \eqref{E:JRN}.
\end{proof}

The second most important term originates from \eqref{GIM:good}.  Its importance stems from the fact that it contains additional coercivity that we will use to estimate other
error terms below.

\begin{lemma}\label{L:goody}
We estimate \eqref{GIM:good} in two pieces:
\begin{gather}
 B_I := \int_I \iint_{|x-y|\leq e^{J-J_0}R} \tfrac{4}{3} a_{kk}(x-y) |\uhi(x)|^{6}|\uhi(y)|^2 \, dx \,dy\,dt \geq 0, \label{BI defn}
\end{gather}
as $a_{kk}\geq 0$ there, and on the complementary region,
\begin{gather}
\int_I \iint_{|x-y|\geq e^{J-J_0}R} |a_{kk}(x-y)| |\uhi(x)|^{6}|\uhi(y)|^2 \, dx \,dy\,dt \lesssim \tfrac{J_0^2}{J} \bigl( K + N^{-3} \bigr). \label{goody}
\end{gather}
\end{lemma}

\begin{proof}
That $a_{kk}\geq 0$ and hence $B_I\geq 0$ is immediate from \eqref{Lap a}.  Further, by construction, $|a_{kk}| \lesssim J_0(Jr)^{-1}$ when $r\geq e^{J-J_0}R$.
In this way, we see that \eqref{goody} relies only on controlling
\begin{align*}
\int_I \iint_{e^{J-J_0}R\leq |x-y|\leq e^J R } & \frac{J_0 |\uhi(x)|^{6}|\uhi(y)|^2}{J|x-y|} \, dx \,dy\,dt,
\end{align*}
which by \eqref{E:22 bound} is
\begin{align*}
&\lesssim_u J_0 \|\uhi\|_{L^\infty_tL^6_x}^6 \sum_{j=J-J_0}^J (Je^jR)^{-1} \cdot  (e^jR)\bigl( K + N^{-3} \bigr) \\
&\lesssim_u \tfrac{J_0^2}{J} \bigl( K + N^{-3} \bigr),
\end{align*}
as needed.
\end{proof}

Now we come to the most dangerous looking term, \eqref{GIM:grad}.  Satisfactory control relies on the full strength of \eqref{Lap a}.

\begin{lemma} Let
$$
\Phi_{jk}(x,y) :=  |\uhi(y)|^{2} \partial_j \overline{\uhi}(x) \partial_k\uhi(x) - (\Im \overline{\uhi}(y) \partial_j \uhi(y))(\Im \overline{\uhi}(x) \partial_k\uhi(x)).
$$
Then
$$
\begin{aligned}
 - \int_I\iint 4a_{jk}(x-y) \Phi_{jk}(x,y) dx\,dy\,dt \lesssim_u  \bigl( \eta^2 + \tfrac{J_0}{J} \bigr)\bigl( K + N^{-3} \bigr) + \tfrac{1}{J_0} B_I.
\end{aligned}
$$
For the $B_I$ notation, refer \eqref{BI defn}.
\end{lemma}

\begin{proof}
As $a_{jk}(x-y)$ is invariant under $x\leftrightarrow y$, we may replace $\Phi$ by the matrix
$$
 \tfrac12 \Phi_{jk}(x,y) + \tfrac12 \Phi_{jk}(y,x),
$$
which is Hermitian-symmetric.  Moreover, for each $x,y$ this matrix defines a positive semi-definite quadratic form on $\R^3$.
To see this, notice that for any vector $e\in\R^3$ and any function $\phi$,
\begin{align*}
\bigl|  e_k e_j  (\Im \bar \phi(y) \phi_j(y))(\Im \bar \phi(x) \phi_k(x)) \bigr|
&\leq  |\phi(y)|\,|e\cdot\nabla \phi(y)|\,|\phi(x)|\,|e\cdot\nabla \phi(x)| \\
&\leq  \tfrac12 |\phi(x)|^2|e\cdot\nabla \phi(y)|^2 + \tfrac12 |\phi(y)|^2|e\cdot\nabla \phi(x)|^2.
\end{align*}
As $a_{jk}$ is a real symmetric matrix (for any $x$ and $y$), its eigenvectors are real. Thus, wherever $a_{jk}$ is positive semi-definite (i.e., $a$ is convex),
the integrand has a favourable sign.  In general, the eigenvalues of the Hessian of a spherically symmetric
function are $a_{rr}$ and $r^{-1}a_r$ with the latter having multiplicity two (ambient dimension minus one).  In our case $a_r\geq0$ and $|a_{rr}|\lesssim J^{-1}r^{-1}$.
Therefore, we are left to estimate
\begin{align}\label{jjkk subgoal}
\int_I\iint_{R<|x-y|<e^JR}  \frac{|\nabla \uhi(x)|^2 |\uhi(y)|^{2}}{J |x-y|} \,dx\,dy\,dt.
\end{align}
To do this, we break the integral into two regions: $|x-x(t)| > R/2$ and $|x-x(t)| \leq R/2$.  In the former case, we use \eqref{E:etas job} and \eqref{E:22 bound}
to obtain the bound
\begin{align*}
&\lesssim_u \|\nabla \uhi\|_{L^\infty_t L^2_x(|x-x(t)|> R/2)}^2  \sum_{j=0}^J (J e^jR)^{-1} \cdot  (e^jR)\bigl( K + N^{-3} \bigr)
    \lesssim_u \eta^2 \bigl( K + N^{-3} \bigr).
\end{align*}
When $|x-x(t)| \leq R/2$, we further subdivide into two regions.  When additionally $|x-y| \geq Re^{J-J_0}$, we estimate in much the same manner as above to obtain the bound
\begin{align*}
&\lesssim_u \|\nabla u\|_{L^\infty_tL^2_x}^2 \sum_{j=J-J_0}^J (J e^jR)^{-1} \cdot  (e^jR)\bigl( K + N^{-3} \bigr)
    \lesssim_u  \tfrac{J_0}{J} \bigl( K + N^{-3} \bigr).
\end{align*}

This leaves us to consider the integral \eqref{jjkk subgoal} over the region where $|x-x(t)| \leq R/2$ and $|x-y| < Re^{J-J_0}$.  Here we use the fact that by the almost periodicity of $u$
(cf. also Remark~\ref{R:small freq} and \eqref{E:etas job}),
$$
\int_{\R^3} |\nabla \uhi(t,x)|^2 \,dx \lesssim_u \int_{\R^3} |\uhi(t,x)|^6 \,dx \lesssim_u \int_{|x-x(t)|\leq R/2} |\uhi(t,x)|^6 \,dx,
$$
uniformly for $t\in[0, \Tmax)$.  We also observe from \eqref{Lap a} that $ J_0 (Jr)^{-1} \leq a_{kk}$; recall $J_0\geq1$.  Therefore, the remaining integral is $\lesssim_u \tfrac{1}{J_0} B_I$.
\end{proof}

The terms appearing in \eqref{GIM:pb} are referred to as momentum bracket terms on account of the notation
\begin{align}
\{\ef,\phi\}_p:=\Re(\ef\nabla\bar{\phi}-\phi\nabla\bar\ef).
\end{align}
Note that applying Proposition~\ref{P:Interact} with $\phi=u_{>N}$ gives $\ef=\ph F(u) - F(\uhi)$.  These error terms are comparatively easy to control:

\begin{lemma}[Momentum bracket terms]\label{L:P brack}
For any $\eps\in(0,1]$,
\begin{equation}
\begin{aligned}\label{E:P bound}
\int_I \int_{\R^3} \int_{\R^3} |\uhi(t,y)|^{2} &\, \nabla a(x-y) \cdot \{\ef,\phi\}_p \,dx\,dy\,dt \\
    &\lesssim_u \eps B_I +  \eta \|\uhi\|_{L^4_{t,x}}^4 + \bigl(\eps^{-1}\eta+\eps\tfrac{J_0^2}{J}\bigr) (N^{-3}+K).
\end{aligned}
\end{equation}
\end{lemma}

\begin{proof}
We begin by expanding the momentum bracket into several terms.  First, we note that
$\{F(\phi),\phi\}_p = -\tfrac23 \nabla |\phi|^6$ and so
\begin{align*}
\{\ef,\uhi\}_{p} &= -\tfrac23 \nabla \bigl(|u|^6 - |\ulo|^6 - |\uhi|^6\bigr) - \{F(u)-F(\ulo),\ulo\}_p - \{\pl F(u), \uhi\}_p.
\end{align*}
Then, using $\{f,g\}_{p}=\nabla(fg)+\Oh(f\nabla g)$, we obtain
\begin{equation}\label{E:P expansion}
\begin{aligned}
\{\ef,\uhi\}_{p} &= \nabla \sum_{j=1}^{5} O(u_{hi}^{j}u_{lo}^{6-j}) + \Oh(u^2 u_{hi} u_{lo}^2 \nabla \ulo) + \Oh(u^3 u_{hi}^2 \nabla \ulo)  \\
&\qquad\qquad + \nabla \Oh( \uhi\pl F(u)) + \Oh(\uhi \nabla \pl F(u)) .
\end{aligned}
\end{equation}
We will treat each of these terms in succession.  The presence of the gradient in front of a term is a signal that we will
integrate by parts in \eqref{E:P bound} before estimating its contribution.

We begin with the first term in \eqref{E:P expansion}.  Integrating by parts and using
$$
\sum_{j=1}^{5} |u_{hi}|^{j}|u_{lo}|^{6-j} \lesssim \eps |u_{hi}|^{6} + \eps^{-1} |u_{lo}|^{2} |u_{hi}| \bigl[ |u_{hi}| + |u_{lo}| \bigr]^3,
$$
we find that we need to obtain satisfactory estimates for
\begin{align}
 \eps \int_I\!\iint |a_{kk}(x-y)| |u_{hi}(t,y)|^2 |u_{hi}(t,x)|^{6} \,dx\,dy\,dt,
\end{align}
which follow already from Lemma~\ref{L:goody}, and for
\begin{align}\label{icky}
 \int_I\!\iint \frac{|u_{hi}(t,y)|^2 |u_{lo}(t,x)|^{2} |u_{hi}(t,x)| [ |u_{hi}(t,x)| + |u_{lo}(t,x)| ]^3}{\eps|x-y|} \,dx\,dy\,dt.
\end{align}
(To obtain this compact form, we use the fact that $|a_{kk}(x-y)|\lesssim |x-y|^{-1}$.)  To bound this second integral, we use the H\"older and Hardy--Littlewood--Sobolev
inequalities, as well as Corollary~\ref{C:all bounds} and \eqref{E:etas job}:
\begin{align*}
\text{\eqref{icky}}
    &\lesssim \eps^{-1} \bigl\| |x|^{-1} * |\uhi|^2 \bigr\|_{L_t^4 L_x^6} \|\uhi\|_{L_t^4 L_x^3} \|\ulo\|_{L_t^4L_x^\infty }^2 \|u\|_{L_t^\infty L_x^6}^3 \\
&\lesssim \eps^{-1} \|\uhi\|_{L_t^\infty L_x^2} \|\uhi\|_{L_t^4 L_x^3}^2 \|\ulo\|_{L_t^4L_x^\infty }^2 \|\ulo\|_{L_t^\infty L_x^6}^3\\
&\lesssim_u \eps^{-1} \eta(N^{-3} + K).
\end{align*}

We now move on to estimating the contribution of the second term in \eqref{E:P expansion}. This is easily estimated using Corollary~\ref{C:all bounds}:
\begin{align*}
\|\Oh(u^2\uhi u_{lo}^2 \nabla \ulo)\|_{L_{t,x}^1}
&\lesssim \|\uhi\|_{L_t^\infty L_x^2}\|\nabla\ulo\|_{L_t^2L_x^6}\|\ulo\|_{L_t^4L_x^\infty}^2 \|u\|_{L_t^\infty L_x^6}^2\\
&\lesssim_u \eta N^{-1}(1+N^3K).
\end{align*}
This takes the desired form when multiplied by
\begin{equation}\label{dumb y}
\int_{\R^3} |\uhi(t,y)|^2\,dy  \lesssim_u \eta^2 N^{-2}.
\end{equation}

To control the third term in \eqref{E:P expansion}, we use Bernstein together with Corollary~\ref{C:all bounds}:
\begin{align*}
\|\Oh(u^3 u_{hi}^2 \nabla \ulo) \|_{L_{t,x}^1} &\lesssim \|\nabla\ulo\|_{L_t^2L_x^\infty } \|\uhi\|_{L_{t,x}^4}^2 \|u\|_{L_t^\infty L_x^6}^{3} \\
&\lesssim_u  N^{1/2} \|\nabla\ulo\|_{L_t^2L_x^6 }  \|\uhi\|_{L_{t,x}^4}^2  \\
&\lesssim_u N^2 \|\uhi\|_{L_{t,x}^4}^4 + N^{-1} (1+N^3K).
\end{align*}

Next, we estimate the contribution from the fourth term in \eqref{E:P expansion}, which, after integration by parts, this takes the form
\begin{equation*}
- \int_I \iint |\uhi(t,y)|^{2} a_{kk}(x-y) \Oh\bigl( \uhi\pl F(u)\bigr)(t,x)  \,dx\,dy\,dt.
\end{equation*}
To continue, we write $\uhi(t,x) = \diverge(\nabla\Delta^{-1}\uhi(t,x))$ and integrate by parts once more.  This breaks the contribution into two parts; after applying H\"older's
inequality and the Mikhlin multiplier theorem, the total contribution is bounded by
\begin{align}\label{51}
& \bigl\| |x|^{-1} * |\uhi|^2 \bigr\|_{L_t^4 L_x^{12}}   \||\nabla|^{-1}\uhi\|_{L_t^2L_x^6} \|\nabla \pl F(u)\|_{L_t^4 L_x^{4/3}} \\
\label{52}
 {} + {}& \bigl\| |x|^{-2} * |\uhi|^2 \bigr\|_{L_t^4 L_x^{12/5}} \||\nabla|^{-1}\uhi\|_{L_t^2L_x^6}  \|\pl F(u)\|_{L_t^4 L_x^{12/5}}.
\end{align}

Applying the Hardy--Littlewood--Sobolev inequality to the first factor in each term and using Sobolev embedding on the very last factor, yields
\begin{align}
\text{\eqref{51}} + \text{\eqref{52}}  &\lesssim \bigl\| |\uhi|^2 \bigr\|_{L_t^4 L_x^{4/3}} \||\nabla|^{-1}\uhi\|_{L_t^2L_x^6} \|\nabla \pl F(u)\|_{L_t^4 L_x^{4/3}}.\label{5.1+2}
\end{align}

To estimate $\nabla \pl F(u)$, we decompose $F(u)=F(\ulo) + \Oh(\uhi u^4)$.  Using H\"older, Bernstein,
and Corollary~\ref{C:all bounds}, we obtain
\begin{align*}
\|\nabla \pl F(\ulo)\|_{L_t^4 L_x^{4/3}} &\lesssim N^{3/4} \|\nabla F(\ulo)\|_{L_t^4 L_x^1} \lesssim N^{3/4} \|\nabla \ulo\|_{L_t^4 L_x^3} \|\ulo\|_{L_t^\infty L_x^6}^4\\
        &\lesssim_u N^{3/4} (1+N^3K)^{1/4}, \\
\|\nabla \pl \Oh(\uhi u^4) \|_{L_t^4 L_x^{4/3}} &\lesssim N^{3/2} \|\uhi u^4\|_{L_t^4 L_x^{12/11}} \lesssim N^{3/2} \|\uhi\|_{L_{t,x}^4} \|u\|_{L_t^\infty L_x^6}^4\\
        &\lesssim_u N^{3/2} \|\uhi\|_{L_{t,x}^4},
\end{align*}
Putting these together with Corollary~\ref{C:all bounds} and \eqref{E:etas job} yields
\begin{align*}
\text{\eqref{5.1+2}} &\lesssim \|\uhi\|_{L_t^\infty L_x^2} \|\uhi\|_{L_{t,x}^4} \||\nabla|^{-1}\uhi\|_{L_t^2L_x^6} \bigl[N^{3/4} (1+N^3K)^{1/4} + N^{3/2} \|\uhi\|_{L_{t,x}^4}  \bigr] \\
&\lesssim_u  \eta N^{-1}  \|\uhi\|_{L_{t,x}^4} N^{-2}(1+N^3K)^{1/2} \bigl[N^{3/4} (1+N^3K)^{1/4} + N^{3/2} \|\uhi\|_{L_{t,x}^4}  \bigr] \\
&\lesssim_u  \eta \bigl[ \|\uhi\|_{L_{t,x}^4}^4  + (N^{-3}+K) \bigr]
\end{align*}

For the fifth (and last) term in \eqref{E:P expansion}, we again write $\uhi = \diverge(\nabla\Delta^{-1}\uhi)$.  After integrating by parts once, the contribution splits
into two pieces, one of which is controlled by \eqref{51} and another which we bound by
\begin{align}\label{53}
\bigl\| (\nabla a) * |\uhi|^2 \bigr\|_{L_t^\infty L_x^\infty} & \||\nabla|^{-1}\uhi\|_{L_t^2L_x^6} \|\Delta \pl F(u)\|_{L_t^2L_x^{6/5}} .
\end{align}
We now decompose $F(u)=F(\ulo) + \Oh(\uhi u_{lo}^2 u^2) +  \Oh(u_{hi}^2 u^3)$.  Using the H\"older and Bernstein inequalities, we deduce
\begin{gather*}
\|\Delta \pl F(\ulo)\|_{L_t^2L_x^{6/5}}\lesssim N\|\nabla \ulo\|_{L_t^2L_x^6} \|\ulo\|_{L_t^\infty L_x^6}^4\lesssim_u N (1+N^3K)^{1/2}, \\
\|\Delta \pl\Oh(u_{hi} u_{lo}^2 u^2)\|_{L_t^2L_x^{6/5}} \lesssim N^2 \|\uhi\|_{L_t^\infty L_x^2} \|\ulo\|_{L_t^4L_x^\infty}^2 \|u\|_{L_t^\infty L_x^6}^2
    \lesssim_u N (1+N^3K)^{1/2},
\end{gather*}
and
\begin{align*}
\|\Delta \pl\Oh(u_{hi}^2 u^3)\|_{L_t^2L_x^{6/5}} \lesssim N^{5/2} \|\uhi\|_{L_{t,x}^4}^2 \|u\|_{L_t^\infty L_x^6}^3 \lesssim_u N^{5/2} \|\uhi\|_{L_{t,x}^4}^2.
\end{align*}
Putting it all together we find
\begin{align*}
\text{\eqref{53}} &\lesssim  \| \uhi \|_{L_t^\infty L_x^2}^2 N^{-2} (1+N^3K)^{1/2} \bigl[N (1+N^3K)^{1/2} + N^{5/2} \|\uhi\|_{L_{t,x}^4}^2  \bigr] \\
&\lesssim_u  \eta^2 \bigl[ \|\uhi\|_{L_{t,x}^4}^4  + (N^{-3}+K) \bigr].
\end{align*}

With the last term estimated satisfactorily, the proof of Lemma~\ref{L:P brack} is now complete.
\end{proof}

Looking back to Proposition~\ref{P:Interact}, we are left with just one term in $\partial_t M(t)$ to estimate, namely, \eqref{GIM:mb}.
As in \cite{CKSTT:gwp}, we call this the mass (Poisson) bracket term and use the notation
$$
\{\ef,\phi\}_m:=\Im(\ef \bar{\phi}).
$$
Notice that $\{|\phi|^4\phi,\phi\}_m=0$ for any function $\phi$.

\begin{lemma}[Mass bracket terms]\label{L:M brack}
For any $\eps>0$,
\begin{equation}\label{M bound}
\begin{aligned}
\biggl| \Im \int_I\int_{\R^3} \int_{\R^3} & \!\{\ef,\uhi\}_{m}(t,y)\nabla a(x-y)\cdot\nabla\uhi(t,x) \overline{\uhi(t,x)} \, dx\, dy\, dt \biggr| \\
    &\lesssim \eta^{1/4} \bigl(\|\uhi\|_{L^4_{t,x}}^4 + N^{-3}+K \bigr).
\end{aligned}
\end{equation}
\end{lemma}

\begin{proof}
Exploiting the cancellation noted above and
$$
F(u) - F(\uhi) - F(\ulo) = \Oh\bigl( \ulo u_\hi u^3 \bigr),
$$
we write
\begin{align}
\{ & \ef, \uhi\}_m = \{\ph F(u) - F(\uhi), \uhi\}_m \notag \\
&=  \{\ph [ F(u) - F(\uhi) - F(\ulo)] , \uhi\}_m - \{\pl F(\uhi) , \uhi\}_m + \{\ph F(\ulo), \uhi\}_m \notag \\
&= \Oh\bigl( \ulo u_\hi^2 u^3 \bigr) - \{\pl F(\uhi) , \uhi\}_m + \{\ph F(\ulo), \uhi\}_m. \label{mass expansion}
\end{align}
We will treat their contributions in reverse order (right to left) since this corresponds to increasing complexity.

The contribution of the third term is easily seen to be bounded by
\begin{align*}
\|\uhi\nabla \uhi\|_{L^\infty_t L^1_x} \|\uhi \ph F(\ulo)\|_{L_{t,x}^1}
&\lesssim \|\nabla \uhi\|_{L_t^\infty L_x^2} \|\uhi\|_{L_t^\infty L_x^2}^2  N^{-1} \|\nabla F(\ulo)\|_{L_t^1L_x^2}\\
&\lesssim_u \eta^2 N^{-3} \|\nabla \ulo\|_{L_t^2L_x^6} \|\ulo\|_{L_t^4L_x^\infty}^2\|\ulo\|_{L_t^\infty L_x^6}^2\\
&\lesssim_u \eta^2 (N^{-3} + K).
\end{align*}

For the second term in \eqref{mass expansion} we write $\uhi = \diverge(\nabla\Delta^{-1}\uhi)$ and integrate by parts.  This yields two
contributions to LHS\eqref{M bound}, which we bound as follows:
\begin{align*}
\|\uhi\nabla \uhi\|_{L^\infty_t L^1_x} & \bigl\| |\nabla|^{-1}\uhi\bigr\|_{L_t^2 L_x^6} \bigl\| \nabla \pl F(\uhi)\bigr\|_{L^2_t L^{6/5}_x} \\
&\lesssim_u \|\uhi\|_{L_t^\infty L_x^2} N^{-2} (1+N^3K)^{1/2} N^{3/2} \|F(\uhi)\|_{L_t^2 L_x^1}\\
&\lesssim_u \eta (N^{-3}+K)^{1/2} \|\uhi\|_{L_{t,x}^4}^2 \|\uhi\|_{L_t^\infty L_x^6}^3\\
&\lesssim_u \eta \bigl( \|\uhi\|_{L_{t,x}^4}^4 + N^{-3}+K \bigr)
\end{align*}
and
\begin{align*}
\bigl\| |x|^{-1} &* |\uhi\nabla\uhi| \bigr\|_{L_t^4 L_x^{12}} \bigl\||\nabla|^{-1}\uhi\bigr\|_{L_t^2L_x^6}  \bigl\| \pl F(\uhi)\|_{L_t^{4}L_x^{4/3}} \\
&\lesssim  \|\nabla \uhi\|_{L_t^\infty L_x^2} \|\uhi\|_{L_{t,x}^4} N^{-2}(1+N^3K)^{1/2} N^{3/4} \|  F(\uhi)\|_{L_t^4L_x^1} \\
&\lesssim_u  \|\uhi\|_{L_{t,x}^4} N^{1/4} (N^{-3}+K)^{1/2} \|\uhi\|_{L_{t,x}^4} \|\uhi\|_{L_t^\infty L_x^6}^{15/4} \|\uhi\|_{L_t^\infty L_x^2}^{1/4} \\
&\lesssim_u  \eta^{1/4} \bigl( \|\uhi\|_{L_{t,x}^4}^4 + N^{-3}+K \bigr).
\end{align*}

We now move to the first term in \eqref{mass expansion}.  This term, or more precisely, the term $\Oh(\ulo u_\hi^5)$ contained therein, is the reason we needed
to introduce the spatial truncation on $a$.  Using $Re^J=N^{-1}$, we estimate this term via
\begin{align*}
\|\nabla\uhi\|_{L^\infty_t L^2_x} \|\uhi\|_{L^4_{t,x}} & \|\nabla a\|_{L^\infty_t L^4_x}  \|\uhi\|_{L_{t,x}^4}^2 \|\ulo\|_{L_t^4 L_x^\infty} \|u\|_{L_t^\infty L_x^6}^3 \\
&\lesssim_u \|\uhi\|_{L^4_{t,x}}^3 (e^JR)^{3/4} \eta^{1/2} ( 1+N^{3}K)^{1/4} \\
&\lesssim_u  \eta^{1/2} \bigl(\|\uhi\|_{L^4_{t,x}}^4 + N^{-3} + K \bigr).
\end{align*}
This completes the control of the mass bracket terms.
\end{proof}

We are now ready to complete the

\begin{proof}[Proof of Theorem~\ref{T:flim}]
From H\"older's inequality, we see that when $\phi=\uhi$ and $a$ is as above, the interaction Morawetz quantity defined in \eqref{E:IM defn} obeys
$$
\sup_{t\in I} |M(t)| \leq 2 \|\uhi\|_{L^\infty_tL^2_x(I\times\R^3)}^3 \|\nabla\uhi\|_{L^\infty_tL^2_x(I\times\R^3)} \lesssim_u \eta^3 N^{-3},
$$
provided, of course, that $N$ is small enough so that \eqref{E:etas job} holds.  Applying the Fundamental Theorem of Calculus to the identity
in Proposition~\ref{P:Interact} and putting together all the lemmas in this section, we reach the conclusion that
\begin{align*}
8\pi \|\uhi\|_{L^4_{t,x}(I\times\R^3)}^4  + B_I &\lesssim_u (\eps+\tfrac1{J_0}) B_I + \eta^{\frac14} \|\uhi\|_{L^4_{t,x}(I\times\R^3)}^4 \\
&\quad +   \bigl( \eta^{\frac14} + \tfrac\eta\eps + \tfrac{J_0^2}{J} + \eta^2 \tfrac{e^{2J}}J\bigr) \bigl(N^{-3} + K \bigr).
\end{align*}
We remind the reader that this estimate is uniform in $\eps,\eta\in(0,1]$, but was derived under several overarching hypotheses: \eqref{E:etas job}, $NRe^J=1$, and $J\geq 2J_0\geq2$.

We now choose our parameters as follows: First $\eps$ and $J_0^{-1}$ are made small enough so that the $B_I$ term on the RHS can be absorbed by that on the LHS.  Next $\eta$ and $J^{-1}$ are
chosen small enough both to handle the $L^4_{t,x}$ on the RHS and to ensure that the prefactor in front of $(N^{-3}+K)$ is smaller than $\eta_0$.  We now choose $R$ and $N^{-1}$ large
enough so that \eqref{E:etas job} holds and then further increase $N^{-1}$ or $R$ so as to ensure $NRe^J=1$.

To fully justify bringing the two terms across the inequality, we need to verify that they are indeed finite.  This is easily done:
\begin{align*}
\| \uhi \|_{L_{t,x}^4}^4 \lesssim \| |\nabla|^{1/4} u_{\geq N}\|_{L^4_t L^3_x}^4 \lesssim N^{-3} \|\nabla\uhi\|_{L_t^4L_x^3}^4
  &\lesssim_u  N^{-3} + N^{-3}\int_I N(t)^2\,dt,
\end{align*}
by Sobolev embedding, Bernstein, and Lemma~\ref{L:ST-N(t)}.  Similarly,
\begin{align*}
B_I &\lesssim \bigl\| |x|^{-1}*|\uhi|^2 \bigr\|_{L^4_t L^{12}_x} \|\uhi\|_{L^\infty_t L^2_x}^{5/4} \|\uhi\|_{L^{19/3}_t L^{114/7}_x}^{19/4} \\
    &\lesssim \|\uhi\|_{L^4_t L^4_x} \|\uhi\|_{L^\infty_t L^2_x}^{9/4} \|\nabla \uhi\|_{L^{19/3}_t L^{38/15}_x}^{19/4}  \\
    &\lesssim_u  N^{-3} + N^{-3}\int_I N(t)^2\,dt,
\end{align*}
by also using the Hardy--Littlewood--Sobolev inequality.
\end{proof}

\begin{remark}\label{R:dream}
As noted in the course of the proof, the necessity of truncating $a(x)$ stems from our inability to estimate one term.
It would be possible to give a much simpler proof if we could show (a priori) that
\begin{equation}\label{51dream}
\| u_\hi^5 u_\lo \|_{L^1_{t,x}} \lesssim_u N^{-2} + \eta N K,
\end{equation}
for $N$ sufficiently small.  We will now describe what appears to be an intrinsic obstacle to doing this.

With current technology, proving \eqref{51dream} without using the interaction Morawetz identity seems to require
proving that it also holds for almost periodic solutions of the \emph{focusing} equation; however, the static solution
$W$ described in Remark~\ref{R:no better} shows \eqref{51dream} does not hold in that setting. From  \eqref{W hat} and simple arguments,
\begin{align*}
\lim_{N\to0} N^{-1} \int_{\R^3} \bigl[W_{> N}(x)\bigr]^5 W_{\leq N}(x) \,dx &= \lim_{N\to0} N^{-1} \int_{\R^3} W(x)^5 W_{\leq N}(x) \,dx \\
&= \lim_{N\to0} N^{-1}\int_{\R^3} |\xi|^2 |\hat W(\xi)|^2 \varphi(\xi/N) \,d\xi \sim 1.
\end{align*}
As $N(t)\equiv1$, it follows that $K=|I|$ and so $\| W_\hi^5 W_\lo \|_{L^1_{t,x}} \gtrsim N K$ for $N$ small.
\end{remark}

%
%
%
%

\section{Impossibility of quasi-solitons}\label{S:no soliton}

In this section, we show that the second type of almost periodic solution described in Theorem~\ref{T:enemies}, namely, those with $\int_0^{\Tmax}N(t)^{-1}\,dt=\infty$,
cannot exist.  This is because their existence is inconsistent with the interaction Morawetz estimate obtained in the last section.

\begin{theorem}[No quasi-solitons]\label{T:no soliton}
{\hskip 0em plus 1em minus 0em}There are no almost periodic solutions $u:[0, \Tmax)\times\R^3\to \C$ to \eqref{nls} with $N(t)\equiv N_k \geq 1$ on each characteristic interval $J_k\subset [0, \Tmax)$
which satisfy $\|u\|_{L^{10}_{t,x}( [0, \Tmax) \times \R^3)} =+\infty$ and
\begin{align}\label{infinite int}
\int_0^{\Tmax}N(t)^{-1}\,dt=\infty.
\end{align}
\end{theorem}

\begin{proof}
We argue by contradiction and assume there exists such a solution $u$.

First we observe that there exists $C(u)>0$ such that
\begin{align}\label{concentration}
N(t) \int_{|x-x(t)|\leq C(u)/N(t)} |u(t,x)|^4\, dx \geq 1/C(u)
\end{align}
uniformly for $t\in [0, \Tmax)$.  That this is true for a single time $t$ follows from the fact that $u(t)$ is not identically zero.  To upgrade this to a statement uniform in time,
we use the fact that $u$ is almost periodic.  More precisely, we note that the left-hand side of \eqref{concentration} is both scale- and translation-invariant and that the
map $u(t)\mapsto \text{LHS} \eqref{concentration}$ is continuous on $L_x^6$ and hence also on $\dot H^1_x$.

Moreover, by H\"older's inequality,
\begin{align*}
N(t) \int_{|x-x(t)|\leq C(u)/N(t)} |u_{\leq N}(t,x)|^4\, dx \lesssim_u \|u_{\leq N}(t)\|_{L_x^6}^4 \quad \text{for any } N>0,
\end{align*}
uniformly for $t\in [0, \Tmax)$.  Combining this with \eqref{concentration} and Theorem~\ref{T:flim} shows that for each $\eta_0 >0$ there exists some $N=N(\eta_0)$ sufficiently small so that
\begin{align*}
\int_I N(t)^{-1}\, dt \lesssim_u \eta_0 N^{-3} + \eta_0 \int_I N(t)^{-1}\, dt
\end{align*}
uniformly for time intervals $I\subset [0, \Tmax)$ that are a union of characteristic subintervals $J_k$.  In particular, we may choose $\eta_0$ small enough to defeat the implicit
constant in this inequality and so deduce that
\begin{align*}
\int_0^\Tmax N(t)^{-1}\, dt = \lim_{T\nearrow \Tmax} \int_0^T N(t)^{-1}\, dt \lesssim_u 1,
\end{align*}
which contradicts \eqref{infinite int}.
\end{proof}

%
%
%
%

\end{document}